\documentclass[11pt,twoside]{extarticle}   

\usepackage[utf8]{inputenc}   
\usepackage[english]{babel}   
\usepackage[T1]{fontenc}   
\usepackage{lmodern}  
\usepackage{eulervm}  
\usepackage{amstext, amssymb, amsthm, bbm, mathtools, mathrsfs, stmaryrd}   
\usepackage{geometry}   
\usepackage[runin]{abstract}   
\usepackage{enumitem}   
\usepackage{titlesec,titletoc}   
\usepackage{hyperref}   
\usepackage{cleveref}   
\usepackage[dvipsnames]{xcolor}   
\usepackage{fancyhdr} 
\usepackage[margin=11pt,font=sf,labelfont=bf, labelsep=endash]{caption}   
\usepackage{subcaption}   
\usepackage{pgf, tikz}   
\usepackage{dsfont}
\usepackage{tabularx}   
\usepackage{array}   
\usepackage{arydshln}   

\definecolor{Mygreen}{rgb}{0.13,0.7,0.25}
\definecolor{Myblue}{rgb}{0.05,0.61,0.92}
\definecolor{Myred}{rgb}{0.8,0.1,0.1}

\hypersetup{ colorlinks = true, linktocpage = true, linkcolor = Myred, citecolor = Mygreen, filecolor = magenta, urlcolor = cyan }

\numberwithin{equation}{section} 
\allowdisplaybreaks[3]  

\tikzstyle{cbvertex}=[draw, circle, blue, thick, fill=white, text=violet, scale=0.6]
\tikzstyle{crvertex}=[draw, circle, red, thick, fill=white, text=violet, scale=0.6]
\tikzstyle{cvvertex}=[draw, circle, Mygreen, thick, fill=Mygreen!20, text=violet, scale=0.6]
\tikzstyle{covertex}=[draw,circle,orange,thick,fill=orange!20,scale=0.6]
\tikzstyle{cnvertex}=[draw, circle, black, thick, fill=white, text=violet, scale=0.6]
\tikzstyle{cnovertex}=[draw, circle, black, thick, fill=orange!20, text=violet, scale=0.6]
\tikzstyle{cbbvertex}=[draw, circle, blue, thick, fill=blue!20, text=violet, scale=0.6]
\tikzstyle{cbrvertex}=[draw, circle, blue, thick, fill=red!20, text=violet, scale=0.6]
\tikzstyle{crbvertex}=[draw, circle, red, thick, fill=blue!20, text=violet, scale=0.6]
\tikzstyle{crrvertex}=[draw, circle, red, thick, fill=red!20, text=violet, scale=0.6]
\tikzstyle{cvrvertex}=[draw, circle, Mygreen, thick, fill=red!20, text=violet, scale=0.6]
\tikzstyle{sbvertex}=[draw, rectangle, blue, thick, fill=white, text=violet, scale=0.8]
\tikzstyle{srvertex}=[draw, rectangle, red, thick, fill=white, text=violet, scale=0.8]
\tikzstyle{snvertex}=[draw, rectangle, black, thick, fill=white, text=violet, scale=0.8]

\abslabeldelim{ \textbf{--}}

\titleformat{\section}[block]{\bfseries\fontsize{15}{0.25}\fontfamily{lmss}\selectfont}{\thesection.}{1em}{}[] 
\titlespacing*{\section}{0pt}{*8}{12pt}
\titleformat{\subsection}[runin]{\bfseries\fontfamily{lmss}\selectfont}{\thesubsection.}{0.5em}{}[.] 
\titleformat{\subsubsection}[runin]{\itshape}{\thesubsubsection.}{0.5em}{}[.] 

\newcommand{\maketitre}{
\begin{center}
\rule{\textwidth}{0.8pt} \\[1.5\baselineskip]
{\fontfamily{qpl}\fontsize{18}{10}\selectfont\bfseries Local Weak Limit of Dynamical Inhomogeneous \\[0.5\baselineskip] Random Graphs} \\[\baselineskip] 
{\fontfamily{qpl}\selectfont\Large\itshape by} \\[\baselineskip]
{\Large {\fontfamily{qpl}\selectfont L\'eo Dort} \textit{\&} {\fontfamily{qpl}\selectfont Emmanuel Jacob}} \\[\baselineskip]
\rule{0.35\textwidth}{0.8pt}\\[2\baselineskip]
\end{center}
}

\newtheoremstyle{my_theorem}{1em}{1em}{\itshape}{}{\bfseries\fontfamily{bch}\selectfont}{.}{ }{\thmname{#1}\thmnumber{ #2}\thmnote{ \sffamily\mdseries(#3)}}   
\newtheoremstyle{my_definition}{1em}{1em}{}{}{\bfseries}{.}{ }{\thmname{#1}\thmnumber{ #2}\thmnote{ (#3)}}   

\theoremstyle{my_theorem}
\newtheorem{thm}{Theorem}[section]   
\newtheorem{prop}[thm]{Proposition}   
\newtheorem{lem}[thm]{Lemma}   
\newtheorem{conj}[thm]{Conjecture}   
\theoremstyle{my_definition}
\newtheorem{defi}[thm]{Definition}   
\newtheorem{rem}[thm]{Remark}   

\newcommand{\N}{\ensuremath{\mathbb{N}}}   
\newcommand{\R}{\ensuremath{\mathbb{R}}}   
\newcommand{\Lp}{\ensuremath{\mathbb{L}}}   
\newcommand{\T}{\ensuremath{\mathbb{T}}}   
\newcommand{\ST}{\ensuremath{\mathbb{ST}}}   
\newcommand{\DG}{\ensuremath{\mathcal{DG}}} 
\newcommand{\DGb}{\ensuremath{\mathcal{DG}_\bullet}} 
\newcommand{\DN}{\ensuremath{\mathcal{DN}}} 
\newcommand{\DNb}{\ensuremath{\mathcal{DN}_\bullet}} 

\newcommand{\E}[1]{\ensuremath{\mathbb{E} \left[#1 \right]}}   
\newcommand{\Prob}[1]{\ensuremath{\mathbb{P} \left(#1 \right)}}   
\newcommand{\bE}{\ensuremath{\mathbb{E}}}   
\newcommand{\bP}{\ensuremath{\mathbb{P}}}   
\newcommand{\convprob}{\ensuremath{\stackrel{(\bP)}{\longrightarrow}}}   
\newcommand{\mubin}{\ensuremath{\mu^{bin}}}   
\newcommand{\mupo}{\ensuremath{\mu^{po}}}   
\newcommand{\1}[1]{\ensuremath{\mathds{1}_{\{ #1 \}}}}   
\newcommand{\de}{\mathrm d}
\newcommand{\eps}{\varepsilon}   

\newcommand{\Fo}{\ensuremath{\mathscr{T}}}   
\newcommand{\Ho}{\ensuremath{\mathscr{H}}}   

\newcommand{\bt}{\ensuremath{\mathsf{t}}}   
\newcommand{\bs}{\ensuremath{\mathsf{s}}}   
\newcommand{\br}{\ensuremath{\mathsf{r}}}   
\newcommand{\bfo}{\ensuremath{\mathsf{f}}}   
\newcommand{\Gr}{\ensuremath{\mathsf{Growth}}}   
\newcommand{\Fu}{\ensuremath{\mathsf{Merge}}}   
\newcommand{\Sp}{\ensuremath{\mathsf{Split}}}   
\newcommand{\Tt}{\ensuremath{T}}   

\begin{document}


\maketitre
\thispagestyle{empty}

\begin{abstract}{
We consider dynamical graphs, namely graphs that evolve over time, and investigate a notion of local weak convergence that extends naturally the usual Benjamini-Schramm local weak convergence for static graphs. One of the well-known results of Benjamini-Schramm local weak convergence is that of the inhomogeneous random graph $IRG_n(\kappa)$ on $n$ vertices with connection kernel $\kappa$. When the kernel satisfies the mild technical condition of being a graphical kernel, the $IRG_n(\kappa)$ converges locally in probability to the unimodular multi-type Poisson-Galton-Watson tree $MPGW(\kappa)$, see the book~\cite{Hofstad_volume_2} for a recent detailed exposure of this result. We extend this to dynamical settings, by introducing the \emph{dynamical} inhomogeneous random graph $DIRG_n(\kappa,\beta)$, with connection kernel $\kappa$ and updating kernel $\beta$, and its limit the \emph{growth-and-segmentation} multi-type Poisson-Galton-Watson tree $GSMPGW(\kappa,\beta)$. 
We obtain similarly the local limit of a variation of our model, namely the \emph{vertex updating} inhomogeneous random graph.
Our framework provides a natural tool for the study of processes defined on these graphs, that evolve simultaneously as the graph itself and with local dynamics. We discuss briefly the case of the contact process, where we obtain a slight reinforcement of the results of~\cite{Jacob-Linker-Morters19,Jacob-Linker-Morters22}.
}
\end{abstract}

\newpage

{\fontfamily{lmss}\selectfont\tableofcontents}

\newpage


\section{Introduction}

\subsection{Motivations}
 
 Local weak limit of graphs has been introduced by Benjamini and Schramm in \cite{Benjamini-Schramm} in the context of random walks on planar graphs. It was then studied in a more general context by Aldous and Steele in \cite{Aldous-Steele}, and by Aldous and Lyons in \cite{Aldous-Lyons}.
 It described the fact that a finite graph, when it is seen from a typical vertex, looks like some limiting graph.
 More formally, we say that a sequence of graphs $(G_n)_{n \geq 1}$ with $n$ vertices \textbf{converges locally weakly} towards an infinite rooted random graph $(G,o)$, if for all distance $d \geq 1$, the ball of radius $d$ centered at $o_n$ uniformly chosen in $G_n$ converges in law to the ball of radius $d$ centered at $o$ in $G$. A classical example is the local weak convergence of an Erd\"os-R\'enyi random graph (in the sparse regime) to a \textbf{Poisson-Galton-Watson tree}, which is a random rooted tree where each vertex has independently a random number of children given by a Poisson random variable.
 More than a convergence in law, Dembo and Montanari proved (in the sparse regime) an almost sure convergence of Erd\"os-R\'enyi random graph towards the Poisson-Galton-Watson tree. See \cite[Proposition 2.6]{Dembo-Montanari}.
 
 Local weak convergence has countless applications in the study of large sparse graphs. Let us just mention that they can be used to study
 their number of spanning trees (see \cite[Theorem 3.3]{Lyons}), or their empirical spectral distribution (see \cite[Corollary 12]{Bordenav-Lelarge-Salez}), as well as the asymptotic properties of processes evolving on these graphs, such as the random walk (see \cite{Benjamini-Schramm}) or the contact process (see \cite{Mountford-Valesin-Yao}).
 
In this work, we are interested in this notion of local weak convergence, but in the context of dynamical graphs, namely we consider graph structures that can evolve in time.
The first study of dynamical graphs we are aware of was given in the field of sociology by Holland and Leinhardt in \cite{Holland-Leinhardt} in order to understand the evolution of social networks. A lot of works on dynamical graphs have also been done in physics, under the common denomination of \textit{temporal networks} or \textit{time-varying networks}, and in epidemiology, for modeling spread of information or infection on social networks. 
In the mathematics literature, there has been an increased interest in various models of dynamical graphs in the last ten years or so. 
However in the mathematics literature the simplest and most studied model of dynamical graph is still \textbf{dynamical percolation}, as introduced by H\"aggstr\"om, Peres and Steif in \cite{Haggstrom-Peres-Steif}, which we further discuss in this introduction, in the sparse regime.

Fix $\kappa>0$, and consider initially the Erd\"os-R\'enyi graph $ER(n,\kappa /n)$ with $n$ vertices, where each edge is independently \textbf{open} (or present in the graph) with probability $\kappa/n$, and \textbf{closed} (or absent from the graph) with probability $1-\kappa/n$. 
In dynamical percolation, edges are then refreshed at rate $1$; upon refreshing the edge is declared open with probability $\kappa/n$ and closed with probability $1-\kappa/n$.
We write $G^{n,\kappa}_t$ the graph at time $t$ with $n$ vertices.
The dynamical graph is stationary, with invariant measure the law of the Erd\"os-R\'enyi graph $ER(n,\kappa /n)$. Thus at any time $t\ge 0$, $G^{n,\kappa}_t$ converges locally weakly to the Poisson-Galton-Watson tree $PGW(\kappa)$. It is then natural to ask for a convergence result for the whole process, which leads to the following questions that motivated this work:
\begin{center}
\itshape Does the dynamical percolation on the complete graph, in the sparse regime, converge locally weakly to a dynamical version of the Poisson-Galton-Watson tree? If yes, in which sense precisely? What new can we learn from such a convergence result?
\end{center}
At this point the question is of course imprecise, but we can already easily imagine a ``dynamical Poisson-Galton-Watson tree'', or $DPGW(\kappa)$, that could appear as a limiting process.
At time 0, it is simply a $PGW(\kappa)$. Then, the dynamics is as follows:
\begin{itemize}
	\item each edge is cut at rate 1. Upon cutting an edge, we only keep the part of the tree which contains the root.
	\item each vertex, at rate $\kappa$, gains a new child that is the root of a new $PGW(\kappa)$ tree. 
\end{itemize}
Heuristically, the first kind of transitions simply comes from the fact that in the finite Erd\"os-R\'enyi graph, each edge is refreshed at rate 1 and has only probability $\kappa/n$ to stay present after refreshing, which vanishes to 0 when $n$ tends to infinity. The second kind of transition comes from the fact that the dynamics may reveal a new edge incident to any given vertex. However, this new edge will typically lead to a totally new and still unseen part of the graph, with local structure thus approximately described by an independent $PGW(\kappa)$ tree. It can further be checked that the $DPGW(\kappa)$ is a well-defined stationary Markov process, and at this point the reader may already have some ideas on rigorous ways to state and prove a convergence result of dynamical percolation to this limiting process.

However, we immediately question the choice made above to ``only keep the part of the tree which contains the root'' after cutting an edge. This is the natural thing to do if we are only interested in the local structure of the graph around the root at any given time. But we might also want to keep track of the further evolution of the vertices which were at some time in the past in the close neighbourhood of the root. This is in particular important if we are interested in the spread of some process or information that evolves simultaneously as the network dynamics. This consideration leads naturally to the notion of \textit{space-time paths} (namely paths of vertices that are connected through edges successively present in the dynamical graph), as well as the notions of dynamical connected components and dynamical balls, that we develop in Section \ref{section:LWC-dyn-graphs}. Note we thus develop a richer topology which encodes more information of our dynamical graph, as compared to the initial choice. However, when considering this topology, we will have a slightly different limiting process, namely the ``growth-and-segmentation Poisson-Galton-Watson tree'', or $GSPGW(\kappa)$. It is defined similarly as the $DPGW(\kappa)$, except that instead of cutting an edge and only keeping one part of the tree, we keep both parts and just call the edge segmented. See Section \ref{section:GS-tree} for a precise construction of these growth-and-segmentation trees.
We can now formulate a version of our main result:

\begin{prop}[Theorem \ref{thm:main-result-0} in the context of dynamical percolation]
The dynamical percolation $\left(G^{n,\kappa}_t\right)_{t\ge0}$ converges locally in probability to the $GSPGW(\kappa)$.
\end{prop}

We will of course define precisely the meaning of this ``local in probability'' convergence, but it implies in particular that we can couple the dynamical ball of a randomly chosen vertex in $G^{n,\kappa}$ up to arbitrary time $t$ and distance $d$, with that of the $GSPGW(\kappa)$, so that they coincide with high probability as $n$ tends to infinity. We actually also have the following more quantitative coupling result:

\begin{prop}[Theorem \ref{thm:main-result-2} in the context of dynamical percolation]
We can couple the dynamical balls in $G^{n,\kappa}$ and in $GSPGW(\kappa)$ up to time $t$ and distance $d$, so that they coincide with probability at least
$$ 1- 34\, \frac{(1+\kappa+\kappa t)^{2d+1}}{n} \,. $$
\end{prop}

In this result, we can of course let $t$ or $d$ depend on $n$, and for example the coupling succeeds with high probability as $n$ tends to infinity if $d$ is fixed and $t=t(n)=o\big(n^{\frac 1 {2d+1}}\big)$.

In the remainder of this paper, we do not consider dynamical percolation but a more general model where we allow inhomogeneities in the connection probabilities, encoded by a connection kernel $\kappa$ (as in inhomogeneous random graphs), as well as in the speed at which each individual edge is updated, encoded by an updating kernel $\beta$. We also discuss alternative network dynamics with simultaneous updates of the edges incident to the same vertex. 
Finally, we provide an application of our convergence results to the study of a contact process running on our dynamical graphs.

\subsection{Relation with other Works}

We are unaware of any existing works on the notion of local weak convergence for dynamical graphs, but other convergence notions have already been investigated in the context of dynamical graphs: 
\begin{itemize}[label=\textbullet]
	\item \textit{Graph limits}, which were introduced in \cite{Lovasz-Szegedy}, look at the limiting homomorphism densities of finite subgraphs.
	Some attempts to include dynamic random graphons has been made for instance in \cite{Rath,Crane,Crane2}. In \cite{Crane}, Crane builds a projection of a process of graphs into the space of graph limits, and he shows that some properties pass through to the limiting process such as the Markov property. More recently, in \cite{Athreya-Hollander-Rollin,Braunsteins-Hollander-Mandjes} also studied graphon-valued processes for which they derived convergence in distribution to \textit{graphon-valued diffusions} that are graphon-valued processes whose dynamic is given by some diffusion model. In the latest work, sample paths large deviations principles are also investigated.
	\item \textit{Scaling limits} look at the convergence of graphs from a global point of view.
	In \cite{Garban-Pete-Schramm}, Garban, Pete and Schramm identify the scaling limit of dynamical percolation for critical percolation on the triangular lattice, which can be interpreted as a dynamical percolation on a suitable graph.
	More recently Rossignol proved a scaling limit for the dynamical percolation on critical Erd\"os-R\'enyi random graphs in \cite{Rossignol}, the resulting process is an explicit ``coalescent-fragmentation'' process on a continuous graph.
\end{itemize}

We mentioned an application of our work when considering a contact process on a dynamical graph, where we obtain a slight improvement of \cite{Jacob-Linker-Morters19,Jacob-Linker-Morters22}. However, many other recent works study processes evolving on a dynamical graph, for example \cite{Linker-Remenik,Hilario-Ungaretti-Valesin-Vares}, again in the context or the contact process, or \cite{Peres-Sousi-Steif18,Avena-Guldas-Hofstad-Hollander18,Avena-Guldas-Hofstad-Hollander19,Cai-Sauerwald-Zanetti,Peres-Sousi-Steif20,Hermon-Sousi,Sousi-Thomas,Lellli-Stauffer} which study mixing times and cover times of random walks under different types of dynamics.
The notion of local weak convergence that we developed might arguably be helpful in simplifying some of these works, or strengthening their results.

\subsection{Organisation of the Paper}
 
In Section \ref{section:DIRG}, we introduce the model of interest, namely dynamical inhomogeneous random graphs. 
In Section \ref{section:LWC-dyn-graphs}, we introduce rigorously the local topology of dynamical rooted graphs, including the dynamical connected component of a dynamical rooted graph and the dynamical balls, as well as a variation for marked graphs. We also discuss the subsequent notions of local weak convergence for a deterministic or random sequence of dynamical graphs.
In Section \ref{section:GS-tree}, we introduce the ``growth-and-segmentation multitype Poisson-Galton-Watson'' tree with kernels $\kappa$ and $\beta$, or $GSMPGW(\kappa,\beta)$, which will be the limiting dynamical graph of our sequence of dynamical inhomogeneous random graphs. To this end we first introduce the growth-and-segmentation trees, which are in a sense all the dynamical components of dynamical rooted graphs ``without loops'', in a similar way that trees are graphs without loops.  
In Section \ref{section:main-results}, we formally state our main theorems of local weak convergence. These include a quantitative result stating until which time and which distance the dynamical ball of a dynamical inhomogeneous random graph can be coupled to the corresponding $GSMPGW(\kappa,\beta)$ when you authorize only finitely many types of vertices.
Section \ref{section:proof} contains the technical proofs of these results. 
In Section \ref{section:vertex-updating}, we discuss an adaptation of our results to a different family of dynamically evolving graphs, in which the dynamics is governed by clocks on each vertex, and in which the vertices incident to that vertex are all updated simultaneously.
In Section \ref{section:CP}, we provide a simple application of our theory of local weak convergence, in the context of a process evolving on these dynamical graphs, namely the contact process, improving the results provided in these settings by~\cite{Jacob-Linker-Morters19} and~\cite{Jacob-Linker-Morters22}. 
Finally,
the \hyperref[appendix:coupling-MP]{Appendix} details a technical coupling result for Markov processes used in our proofs.

\subsection*{Notations}

All unspecified limits are as $n\to\infty$. \\
We say that a sequence of events $(A_n)_{n\ge1}$ occurs \emph{with high probability} (\textit{w.h.p.}) if $\Prob{A_n} \to 1$. 
Convergence in probability is denoted by $\convprob$. \\
We denote $\mupo_\lambda$ the Poisson law with parameter $\lambda\ge0$, $\mubin_{k,p}$ the binomial law with parameters $k\in\N$ and $0\le p\le 1$, and $\mu$ for a general probability measure. \\
We reserve the variables $u,v,w$ for vertices, $x,y,z$ for the type of vertices in $S$.
We also denote $t,s,r$ for positive times, and $\bt,\bs,\br$ for (multitype) trees. \\
We also denote $V_n = \{1,\dots,n\}$.


\section{Dynamical Inhomogeneous Random Graphs}\label{section:DIRG}

 The model of \textbf{inhomogeneous random graphs} was introduced in fairly general setting by Bollob\'as, Janson and Riordan in \cite{Bollobas-Janson-Riordan}, and is one of the main object of study of the book of van der Hofstad \cite{Hofstad_volume_2}.
 We enrich this model with a dynamic to define the \textbf{dynamical inhomogeneous random graph} $\big(G^{\mathcal V,n,\kappa,\beta}_t\big)_{t\ge0}$ with vertex set $V_n$.
 Here,
 $$ \mathcal V = \big(S,\mu,(x^{(n)}_1,\dots,x^{(n)}_n)_{n\ge1}\big) \, , $$
 is the \textbf{vertex space}, $(S,d_S)$ is a separable metric space called \textbf{state space} or \textbf{mark space}, and endowed with a Borel probability measure $\mu$, and $x^{(n)}_i\in S$ is the (possibly random) mark associated with vertex $i\in V_n$.  No relationship is assumed between $x^{(n)}_i$ and $x^{(n')}_i$, but to simplify notations we shall write further $(x_1,\dots,x_n) = (x^{(n)}_1,\dots,x^{(n)}_n)$. It is assumed that the empirical measure $\displaystyle \nu_n:=\frac 1 n \sum_i \delta_{x_i}$ converges to the measure $\mu$ in the usual space of probability measures on $S$ (and this convergence holds in probability when the $x_i$ are random). Equivalently,  for any $\mu$-continuity set $A\subseteq S$,
 \begin{equation}\label{eq:nu_n}
 \nu_n(A) := \frac{\#\{ i\in V_n: x_i\in A \}}{n} \convprob \mu(A) \, .
 \end{equation}
 
 From now on, we assume for convenience that $S$ is \textbf{compact} and \textbf{ordered}.
 We will be particularly interested in the case when $S$ is $[0,1]$, endowed with the Lebesgue measure, and $x^{(n)}_i=i/n$, as well as in the case $S$ finite (as introduced and studied by S\"oderberg in \cite{Soderberg,Soderberg2003a,Soderberg2003b,Soderberg2003c}).

 Given the sequence $x_1,\dots,x_n$, $\big(G^{\mathcal V,n,\kappa,\beta}_t\big)_{t\ge0}$ is the random dynamical graph on $V_n$, such that any two vertices $u$ and $v$ are initially connected by an edge independently of the others and with a probability
 $$ p_{u,v} = \frac{1}{n} \kappa(x_u,x_v) \wedge 1 \,, $$
 and the connection is updated at rate $\beta_{u,v} = \beta(x_u,x_v)$.
 That is, after an exponential time with expectation $1/\beta_{uv}$, the vertices are again connected with probability $p_{uv}$, independently of the past and independently of the other edges. See Fig. \ref{fig:edge_updating}.
 
 \begin{figure}[!ht]
 \centering
 \begin{tikzpicture}[scale=1.1]
 \path[thick] ({360/10 * (1 - 1)}:1.75) edge [bend right] ({360/10 * (2 - 1)}:1.75);
 \draw[thick] ({360/10 * (1 - 1)}:1.75) -- ({360/10 * (3 - 1)}:1.75);
 \draw[thick] ({360/10 * (3 - 1)}:1.75) -- ({360/10 * (4 - 1)}:1.75);
 \draw[thick] ({360/10 * (3 - 1)}:1.75) -- ({360/10 * (7 - 1)}:1.75);
 \draw[thick] ({360/10 * (3 - 1)}:1.75) -- ({360/10 * (8 - 1)}:1.75);
 \path[thick] ({360/10 * (6 - 1)}:1.75) edge [bend left] ({360/10 * (9 - 1)}:1.75);
 \path[thick] ({360/10 * (7 - 1)}:1.75) edge [bend right] ({360/10 * (8 - 1)}:1.75);
 \path[thick] ({360/10 * (9 - 1)}:1.75) edge [bend right] ({360/10 * (10 - 1)}:1.75);
 \draw[red,thick,dotted] ({360/10 * (5 - 1)}:1.75) -- ({360/10 * (10 - 1)}:1.75);
 \foreach \s in {1,...,10}
 { \draw [black, thick, fill=white] ({360/10 * (\s - 1)}:1.75) circle (2.5pt); }
 \draw [red, thick, fill=red!20] ({360/10 * (5 - 1)}:1.75) circle (2.5pt);
 \draw [red, thick, fill=red!20] ({360/10 * (10 - 1)}:1.75) circle (2.5pt);
 
 \node [] at (3.5,0) {\huge $\longrightarrow$};
 
 \begin{scope}[xshift=7cm]
 \path[thick] ({360/10 * (1 - 1)}:1.75) edge [bend right] ({360/10 * (2 - 1)}:1.75);
 \draw[thick] ({360/10 * (1 - 1)}:1.75) -- ({360/10 * (3 - 1)}:1.75);
 \draw[thick] ({360/10 * (3 - 1)}:1.75) -- ({360/10 * (4 - 1)}:1.75);
 \draw[thick] ({360/10 * (3 - 1)}:1.75) -- ({360/10 * (7 - 1)}:1.75);
 \draw[thick] ({360/10 * (3 - 1)}:1.75) -- ({360/10 * (8 - 1)}:1.75);
 \path[thick] ({360/10 * (6 - 1)}:1.75) edge [bend left] ({360/10 * (9 - 1)}:1.75);
 \path[thick] ({360/10 * (7 - 1)}:1.75) edge [bend right] ({360/10 * (8 - 1)}:1.75);
 \path[thick] ({360/10 * (9 - 1)}:1.75) edge [bend right] ({360/10 * (10 - 1)}:1.75);
 \draw[red,thick] ({360/10 * (5 - 1)}:1.75) -- ({360/10 * (10 - 1)}:1.75);
 \foreach \s in {1,...,10}
 { \draw [black, thick, fill=white] ({360/10 * (\s - 1)}:1.75) circle (2.5pt); }
 \draw [red, thick, fill=red!20] ({360/10 * (5 - 1)}:1.75) circle (2.5pt);
 \draw [red, thick, fill=red!20] ({360/10 * (10 - 1)}:1.75) circle (2.5pt);
 \end{scope}
 \end{tikzpicture}
 \caption{One move of the dynamic of the dynamical inhomogeneous random graph. The red dotted edge which is absent on the left is the one chosen to be updated.}\label{fig:edge_updating}
 \end{figure}
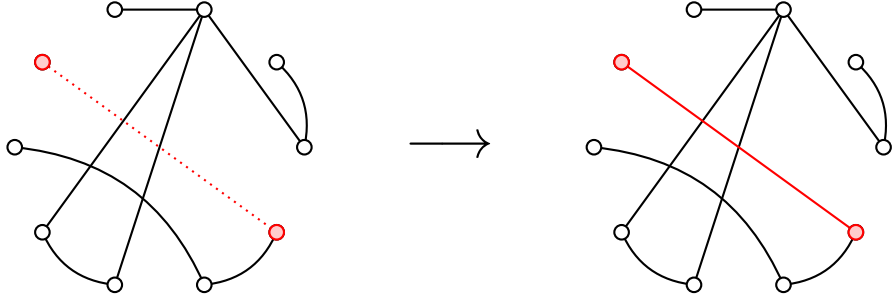

 Here $\kappa:S\times S\to[0,\infty)$ and $\beta:S\times S\to[0,\infty)$ are symmetric non-negative measurable functions on $S\times S$ called the \textbf{connection kernel} and the \textbf{updating kernel} respectively. We will often drop the vertex space from the notation and write $G^{n,\kappa,\beta}$. Let us note that:
 \begin{itemize}
 \item The dynamical networks we consider are \emph{stationary}. At time $t=0$ and at later times, the law of the network $G^{n,\kappa,\beta}_t$ is that of the inhomogeneous random graph with vertex space $\mathcal V$ and kernel $\kappa$, where any two vertices are connected independently with probability $p_{uv}$.
 \item Each edge evolves independently according to a $2$-state Markov chain, where the edge $uv$ turns open at rate $\beta_{uv} p_{uv}$ and turns closed at rate $\beta_{uv} (1-p_{uv})$. If we consider another dynamical inhomogeneous random graph with the same vertex space but different kernels $\kappa'\le\kappa$ and $\beta'\le\beta$, there is a natural coupling for the evolution of the edges in the two dynamical graphs, where the updating times for $G^{n,\kappa',\beta'}$ are included in those for $G^{n,\kappa,\beta}$, and, initially or when the two graphs update simultaneously, the updating edge is open in $G^{n,\kappa,\beta}$ if it is in $G^{n,\kappa',\beta'}$. Using this coupling, we can see that the evolution of the edge $uv$ coincide in the two graphs on the whole time interval $[0,T]$ with probability at least
 \begin{equation}\label{eq:approximation-of-kernel}
 1-(p_{uv}-p'_{uv}) - (\beta_{uv} p_{uv} - \beta'_{uv} p'_{uv}) T.
 \end{equation}
 \end{itemize}
 
 We will of course be interested in the study of the dynamical inhomogeneous random graph in the limit $n\to\infty$.
 We will always assume that the kernels $\kappa$ and $\beta$ are \emph{graphical} as in the following definition.
 
 \begin{defi}
 Let $\mathcal V =(S,\mu,(x_1,\dots,x_n)_{n\ge1})$ be a vertex space.
 The kernels $\kappa$ and $\beta$ are called \textbf{graphical} on $\mathcal V$ if
 \begin{enumerate}
 \item\label{def:graphical-i} $\kappa$ and $\beta$ are continuous \textit{a.e.} on $S\times S$;
 \item\label{def:graphical-ii} $\kappa\in\Lp^1(S\times S,\mu\otimes\mu)$ and $\kappa \beta \in\Lp^1(S\times S,\mu\otimes\mu)$;
 \item\label{def:graphical-iii} The kernel $\kappa$ satisfies:
 \begin{equation}\label{eq:graphical}
 \frac{1}{n^2} \sum_{1\le u<v\le n} \kappa(x_u,x_v) \wedge n \convprob \frac{1}{2} \int\int \kappa(x,y) \, \mu(\de x)\mu(\de y) \, 
 \end{equation}
 \item\label{def:graphical-iv} The kernels $\kappa$ and $\beta$ satisfy:
 \begin{equation}\label{eq:graphical-2}
 \frac{1}{n^2} \sum_{1\le u<v\le n} \beta(x_u,x_v) \left(\kappa(x_u,x_v) \wedge n\right) \convprob \frac{1}{2} \int\int \beta(x,y)\kappa(x,y) \, \mu(\de x)\mu(\de y) \, .
 \end{equation}
 \end{enumerate}
 \end{defi}
 
 Note that this is a slight adaptation of the definition of a graphical kernel in \cite{Bollobas-Janson-Riordan} in the context of dynamical graphs with the additional updating kernel $\beta$.
 Condition \ref{def:graphical-iii} implies that the number of edges is asymptotically proportional to $n$, with proportionality constant being precisely $\frac{1}{2} \int\int \kappa(x,y) \mu(\de x)\mu(\de y)$.
 Condition \ref{def:graphical-iv} implies a similar result when considering all the edges present in the network in a finite time interval $[0,T]$.
 This means in particular that the model is \textbf{sparse}.

 \begin{rem}
 \begin{enumerate}
 \item The definition of a graphical kernel in~\cite{Bollobas-Janson-Riordan} does not request~\eqref{eq:graphical} but the weaker assumption that the expectation of the LHS converges to the RHS, however they also show that the two are actually equivalent. We prefer the maybe more meaningful definition with a convergence in probability. The reader will not lose much if he assumes for simplicity that the $x_i$ are deterministic and if he replaces the convergence in probability by a simple convergence in~\eqref{eq:nu_n}, \eqref{eq:graphical} and~\eqref{eq:graphical-2}.
 \item For convenience we supposed that the kernels $\kappa$ and $\beta$ do not depend on $n$. However, our results generalize easily to the case where $\kappa$ and $\beta$ can depend on $n$ and be \emph{graphical with limits $\kappa$ and $\beta$}, in the sense that we have $\kappa_n(x,y)\to\kappa(x,y)$ and $\beta_n(x,y)\to\beta(x,y)$ for \textit{a.e.} $(x,y)\in S\times S$, as well as an adaptation of \eqref{eq:graphical} and \eqref{eq:graphical-2} with of course $\kappa_n$ and $\beta_n$ in the LHS.
 
 \item When $S=\{1,\dots,r\}$ is finite, the vertices can be classified according to $r$ different types, and we say $\kappa$ and $\beta$ are \textbf{finite-type kernels}. Our strategy in the following will be to prove local convergence results for the dynamical inhomogeneous random graphs, by first approximating the kernels by finite-type ones (or regular finitary kernels to be more precise, see later).
 \end{enumerate}
 \end{rem}


\section{Local Convergence for Dynamical Graphs}\label{section:LWC-dyn-graphs}

In this part we introduce the space of (possibly marked) dynamical graphs that we consider, and the associated notions of local convergence and local weak convergence. These are similar to the notions of local convergence and Benjamini-Schramm local weak convergence for static graphs, of course adapted to our dynamical settings.

\subsection{Dynamical Graphs and Networks}

 A \textbf{dynamical rooted graph} is a process of rooted graphs $(G_t,o)_{t\ge0} = ((V_t,E_t),o)_{t\ge0}$ depending on time $t\ge0$, where the set of vertices $V_t$ always contains a specific vertex, called the \textbf{root} $o$, and is included in a vertex set $V$. We sometimes simply write $G$ or $(G,o)$ for the dynamical graph $(G_t,o)_{t\ge 0}$. We further always assume that our dynamical graph:
 \begin{itemize}
 \item has \textbf{nondecreasing vertex set} $V_t$. Each vertex $v\in V_t$ then has a \textbf{birth-time} $b(v)$ equal to the infimum value of $\{s: v\in V_s\}$, and is smaller than or equal to $t$. We always assume that $v$ belongs to $V_{b(v)}$, which can be seen as a right-continuity property.\\
 We also write $(V_t,E_{\le t}) := (V_t,\cup_{s\le t} E_s)$ the \textbf{accumulated graph} at time $t$, and
 $(V_\infty,E_{<\infty}) := (\cup_{t\ge 0}V_t,\cup_{t\ge 0} E_t)$ the total accumulated graph. We say we have a \textbf{growing graph} if the graph and the accumulated graph coincide for all finite $t$. 
 \item is \textbf{locally finite}, in the sense that for any finite $t\ge 0$, the accumulated graph $(V_t,E_{\le t})$ is locally finite.
 \end{itemize}
 
 When looking at navigability along a dynamical graph, or in the study of processes living on this graph, we have a natural notion of space-time paths along the graph, as well as space-time ``distances'' from the root.
 
 \begin{defi}
 A dynamical rooted graph $(G_t,o)_{t\ge0}$ is called \textbf{connected} if for all $t\ge 0$ and $v$ in $V_t$, there exists $k\ge 0$, a path $o=v_0,\ldots, v_k=v$ in the graph $(V_t,E_{\le t})$ and times $0\le t_1\le\ldots \le t_k\le t$ such that for all $i=1,\ldots, k$, we have $v_{i-1}v_i \in E_{t_i}$. \\
 We write $\overset{\rightarrow}{\de}_t(v)$ for the smallest possible value for $k$ above (which can be infinite if the dynamical graph is not connected).
 \end{defi}
 
 Note that a dynamically connected graph $(G_t,o)_{t\ge0}$ can fail to be a connected graph for fixed $t>0$.
 Note also that the function $t\mapsto \overset{\rightarrow}{\de}_t(v)$, defined on $[b(v),+\infty)$, is nonincreasing, and takes finite values if the dynamical graph is connected. We further define the dynamical component of the root, as well as the dynamical balls centered at the root.
 
 \begin{defi}
 Let $(G_t,o)_{t\ge0}$ be a dynamical rooted graph on a vertex set $V$.
 \begin{enumerate}[label=(\alph*)]
 \item The \textbf{dynamical component} of the root $o$ in $(G_t)_{t\ge0}$ is the largest connected dynamical subgraph $(G^{[\infty]}_t,o)_{t\ge0} = ((V^{[\infty]}_t(G),E^{[\infty]}_t(G)),o)_{t\ge0}$, characterized for $t\ge0$ by
 $$ V^{[\infty]}_t(G) = \left\{ v\in V_t: \: \overset{\rightarrow}{\de}_t(v)<+\infty \right\} $$
 and
 $$ E^{[\infty]}_t(G) = \left\{ u\sim v\in E_t: \: u,v\in V^{[\infty]}_t(G) \right\} \,. $$
 \item The \textbf{dynamical ball of radius} $d$ centered in the root $o$ in $(G_t)_{t\ge0}$ is the dynamical rooted graph $(G^{[d]}_t,o)_{t\ge0} = ((V^{[d]}_t(G),E^{[d]}_t(G)),o)_{t\ge0}$, where for all $t\ge0$,
 $$ V^{[d]}_t(G) = \left\{ u\in V_t: \: \overset{\rightarrow}{\de}_t(u)\le d \right\} $$
 and
 $$ E^{[d]}_t(G) = \left\{ u\sim v\in E_t: \: u,v\in V^{[d]}_t(G) \right\} \,. $$
 \end{enumerate}
 \end{defi}
 
 Let us stress again that in the dynamical notion of connectivity that we have introduced, the graphs $G^{[\infty]}_t$ and $G^{[d]}_t$ can fail to be connected at a fixed time $t$.
 
 When there is no confusion, we write in the following $(G,o)$ for the dynamical graph $(G_t,o)_{t\ge0}$.

 It is common to identify two graphs that are isomorphic. In our dynamical settings, we say two dynamical rooted graphs $(G^1,o_1)$ and $(G^2,o_2)$ are \textbf{isomorphic}, and we also write $(G^1,o_1) \equiv (G^2,o_2)$, if there exists a bijective map $\gamma:V_\infty(G^1) \to V_\infty(G^2)$ such that $\gamma((G^1,o_1))=(G^2,o_2)$, where $\gamma((G^1,o_1))$ is by definition the dynamical rooted graph
 $$ \gamma \big( (G^1,o_1) \big) := \big( (\gamma(V^1_t), \gamma(E^1_t)),\gamma(o_1) \big)_{t\ge 0} \,. $$ 
 In other words the bijective map $\gamma$ preserves:
 \begin{itemize}
 \item the root, that is $\gamma(o_1) = o_2$ ; and
 \item the dynamic of edges, that is for all time $t\ge0$, $\gamma(V^1_t) = V^2_t$ and $\{u,v\}\in E_t(G^1)$ precisely when $\{\gamma(u),\gamma(v)\}\in E_t(G^2)$.
 \end{itemize}
 It is not difficult to show that $\equiv$ is an equivalence relation on the set of connected dynamical rooted locally finite graphs with nondecreasing vertex set.
 
 \begin{defi}
 We write $\DGb$ for the space of isomorphism classes of connected dynamical rooted locally finite graphs, with nondecreasing vertex set.
 \end{defi}
 
 \begin{rem}
 Note that the map $\gamma$ is not allowed to depend on $t$, and as such, an element of $\DGb$ cannot be identified with a function on the space of isomorphic classes of connected (\textit{static}) rooted locally finite graphs. A way to by pass through this difficulty would be:
 \begin{itemize}
 \item to consider the edges in 
 $$ \left(\cup_{s\le t} E^{[\infty]}_s\right)\backslash E_t $$
 as ``segmented edges'', so that for fixed $t$, the graph $\left(V^{[\infty]}_t, \cup_{s\le t} E^{[\infty]}_s\right)$ containing also the segmented edges is now connected.
 \item and then to identify unambiguously every single edge by defining a cyclic ordering of the neighbours of any given vertex, which would further respect the order of appearance of the edges in the graph (given that the edges incident to a given vertex in the dynamical component is now a growing set when we also consider the segmented edges).
 \end{itemize}
 We refrain ourselves to do so at this point, for the two reasons that this construction is not that natural for a given dynamical graph (in particular the requested ordering of the edges incident to each vertex), and that the space $\DGb$ is  actually a perfectly nice working space to define a local topology, as we do in the next section.\\
 However, our main theorem involves the Growth-and-Segmentation Multitype Poisson-Galton-Watson tree, which is best constructed as a Markov process on the space of ordered segmented trees. Of course, this process can then also be seen as a random element of $\DGb$.
 \end{rem}
 
 We will also consider dynamical graphs with the additional structure of a mark associated to each vertex, living in a \textbf{mark space} $S$, which we recall is assumed to be compact and ordered\footnote{In this section, it would actually suffice to suppose $S$ is a Polish space. The order on $S$ will be important when defining the growth-and-segmentation Poisson-Galton-Watson tree in next section, while the compactness of $S$ will be used in the approximation argument of general kernels by regular finitary kernels in Section \ref{section:proof}.}.
 
 A \textbf{dynamical marked graph}, or \textbf{dynamical network}, is then $(G,m)$, where $G$ is a dynamical graph and $m$ the mark function which belongs to $S^{V_\infty(G)}$.
 Two rooted dynamical networks $(G^1,o_1, m_1)$ and $(G^2,o_2,m_2)$ are called isomorphic if there is a bijective map $\gamma:V^1_\infty \to V^2_\infty$ such that $\gamma(G^1,o_1,m_1)= (G^2,o_2,m_2)$, where now
 $$ \gamma(G^1,o_1,m_1)=(\gamma(G^1),\gamma(o_1), m_1\circ \gamma^{-1}) \, . $$
 In other words, two dynamical networks are isomorphic if the dynamical graphs are isomorphic, with associated map $\gamma$ preserving the marks of the vertices.
 
 \begin{defi}
 We write $\DNb^S$ for the space of isomorphism classes of connected dynamical rooted locally finite networks, with nondecreasing vertex set and marks in space $S$. We also write $\DNb$ when the mark space is implicit.
 \end{defi}

\subsection{Local Convergence}

 We now endow the spaces $\DGb$ and  $\DNb^S$ with the topology of \textbf{local convergence}. We should again note the similarity with the usual notion of local convergence for rooted (static) graph. We should also note that we take into account the time dynamics in a way which is similar to the notion of uniform convergence on compact sets, but with the added subtlety we already mentioned that an element of  $\DGb$ cannot be assimilated to a function on the space of rooted graphs.
 
 \begin{defi} \label{def:deterministic-loc-conv}
 We say a sequence of dynamical rooted graphs $\{(G^n_t,o_n)_{t\ge0},\, n\ge1\}$ \textbf{converges locally to} $(G^\infty_t,o_\infty)_{t\ge0}$ \textbf{in} $\DGb$ if for all $d\ge 1$ and $T<+\infty$, for $n$ large enough,
 $$ (G^{n,[d]}_t, o_n)_{t\le T} \equiv (G^{\infty,[d]}_t, o_\infty)_{t\le T} \,. $$
 We say a sequence of dynamical networks $(G^n, m_n)$ \textbf{converges locally to} $(G^\infty,m_\infty)$ \textbf{in} $\DNb^S$ if for all $d\ge 1$, for all $T<+\infty$ and $\delta>0$, for $n$ large enough, there is a bijective map 
 $\gamma_n: V^{n,[d]}_T \to V^{\infty,[d]}_T$ such that  
 $$ \left\{
 \begin{array}{l}
 \gamma_n\left((G^{n,[d]}_t, o_n)_{t\le T} \right) =(G^{\infty,[d]}_t, o_\infty)_{t\le T},
 \\
 \forall v\in V^{n,[d]}_T, \,d_S(m_\infty(\gamma_n(v)),m_n(v))\le \delta \,.
 \end{array}
 \right. $$
 \end{defi}
 
  In words, we ask the dynamical graphs to be  isomorphic when restricted to graph distance $d$ and time $T$, and in the case of dynamical networks we further request the associated map to leave the marks at distance less than $\delta$ from each other.
 
 \begin{prop}\label{prop:completness-of-DG}
 The topology of local convergence on $\DGb$ is metrizable by the distance $D_\bullet$ defined as
 $$ D_\bullet\left((G^1,o_1),(G^2,o_2)\right)
 = \sum_{k=0}^\infty \frac{1}{2^k} D_k\left((G^1,o_1),(G^2,o_2)\right) \,, $$
 where
 $$ D_k\left((G^1,o_1),(G^2,o_2)\right)
 = \left( 1+\sup\left\{ d\ge1, \: (G^{1,[d]}_t, o_1)_{0\le t\le k}\equiv(G^{2,[d]}_t, o_2)_{0\le t\le k} \right\} \right)^{-1} \,. $$
 Moreover, $(\DGb,D_\bullet)$ is a complete ultrametric space.
 
The topology of local convergence on $\DNb^S$ can be defined by the distance $D_\bullet^S$ defined as
 $$ D_\bullet^S\left((G^1,o_1,m_1),(G^2,o_2,m_2)\right)
 = \sum_{k=0}^\infty \frac{1}{2^k} D_k^S\left((G^1,o_1,m_1),(G^2,o_2,m_2)\right) \,, $$
 where $D_k$ is defined by
 $$ D_k^S\left((G^1,o_1,m_1),(G^2,o_2,m_2)\right)
 = \frac{1}{1+R_k^S} \,, $$
 and $R_k^S$ by
 $$ R_k^S=\sup\left\{ d\ge1 : \exists \gamma :
 \begin{array}{l}
 \gamma\left( (G^{1,[d]}_t, o_1)_{0\le t\le k}\right)=(G^{2,[d]}_t, o_2)_{0\le t\le k} ,\\
 \forall v\in V^{1,[d]}_k, d_S(m_2(\gamma(v)),m_1(v))\le 1/d.
 \end{array}
 \right\} \,. $$
 \end{prop}
 
 We omit the proof of this proposition, as it essentially follows from a careful rereading of \cite[Appendix A]{Hofstad_volume_2}. 
 Note that the metric spaces $(\DGb,D_\bullet)$ and $(\DNb^S,D_\bullet)$ fail to be separable, for the same reason that the space of real-valued functions endowed with the topology of uniform convergence on compact sets fails to be separable, and still fails to be separable when restricted to continuous or c\`adl\`ag functions. 
 
 It is natural to ask whether we also have a notion of local Skorokhod topology, which would be similar to the $J1$-Skorokhod topology and define a Polish space. It is indeed the case, as we now explain in the context of unmarked graphs.
 
 Consider $\mathcal{CDG}_\bullet\subset\mathcal{DG}_\bullet$ the subset of \textit{c\`adl\`ag} rooted graphs, in the sense that for every finite time $t$ and distance $d$, the process $(G^{[d]}_s,o)_{0\le s\le t}$ has finitely many jumps and is right continuous. 
 For finite time $T$, define $D_{SK_T}((G^1,o_1),(G^2,o_2))$ as the infimum of all those values of $\delta\in(0,1]$ for which there exists $(\tilde G^1,\tilde o_1)\equiv(G^1,o_1)$, $(\tilde G^2,\tilde o_2)\equiv(G^2,o_2)$ and a grid $0=s_0<\dots<s_k$, with $s_k\ge T$, and $0=t_0<\dots<t_k$ with $t_k\ge T$, such that $|t_i-s_i|\le\delta$ for $i=0,\dots,k$ and with $d=\lceil1/\delta-1\rceil$,
 $$ (\tilde G^{1,[d]}_s,o_1) = (\tilde G^{2,[d]}_t,o_2) \qquad \textrm{for all } i=0,\dots,k-1 \textrm{ and } s_i\le s<s_{i+1} \textrm{ and } t_i\le t<t_{i+1} \,. $$
 Then the \textbf{local Skorokhod topology} is defined by the metric
 $$ D_{Sk}\big((G^1,o_1),(G^2,o_2)\big) := \sum_{k=1}^\infty \frac{1}{2^k} D_{Sk_k}\big((G^1,o_1),(G^2,o_2))\big) \,, $$
 and makes $(\mathcal{CDG}_\bullet, D_{Sk})$ a Polish space. This space would be another natural working space, however we obtain our results directly for the finer local topology, and for this reason we do not work further with the local Skorokhod topology.

\subsection{Local Weak Convergence for Dynamical Graphs}

 In this section, we define an analog of Benjamini-Schramm convergence, or local weak convergence, in the context of dynamical graphs. We could define local weak convergence on either spaces $\DG_\bullet$ for unmarked graphs, or $\DN^S_\bullet$ for marked graphs, but we work here directly with marked graphs.
 Consider a sequence of dynamical graphs $\{G^n,\, n\ge1\}$ with marks in $S$ and defined on a finite vertex set $V_n$, with size tending to infinity. We first consider the case of a \textbf{deterministic} such sequence. For each $n$, we then obtain a random element of $\DNb^S$ by first selecting a root $\mathbf o_n$ uniformly at random on $V_n$, and considering only $(G^{n,[\infty]},\mathbf o_n)$ the dynamical component containing $\mathbf o_n$. To enlighten the notation, we use below a slight abuse of notation and write $(G^n,\mathbf o_n)$ for  $(G^{n,[\infty]}, \mathbf o_n)$, knowing that we always work on $\DNb^S$.
 
 \begin{defi}\label{def:deterministic-rand-loc-conv}
 Given a deterministic sequence $\{ G^n,\, n\ge1 \}$ of marked dynamical graphs on $V_n$, we say that $\{ G^n,\, n\ge1 \}$ \textbf{converges locally weakly on} $\DNb^S$ to a random dynamical rooted graph $(\mathbf G^\infty,\mathbf o)$, if the sequence $\{ (G^n,\mathbf o_n), \, n\ge1 \}$, where $\mathbf o_n$ is a uniform vertex of $V_n$, converges in distribution to $(\mathbf G^\infty,\mathbf o)$ \textit{w.r.t.} the local distance $D_\bullet^S$.
 \end{defi}
 
 Note that even though $G^n$ is here deterministic, we obtain a (nondeterministic) random variable $(G^n,\mathbf o_n)$ by the random choice of the root $\mathbf o_n\in V_n$.
 An equivalent definition is the weak convergence of the probability measures
 $$ \frac{1}{n} \sum_{v\in V_n} \delta_{(G^n,v)} $$
 to the law $\nu$ of $(\mathbf G^\infty,\mathbf o)$, as probability measure on the metric space $(\mathcal{DN}_\bullet^S,D_\bullet^S)$, which can also be restated as the convergence
 $$ \frac{1}{n} \sum_{v\in V_n} h( G^n,v ) \to \E{h( \mathbf G^\infty,\mathbf o )} $$
 for every bounded and continuous function $h:\DNb^S\to\R$.
 
 Having discussed the notion of local weak convergence for deterministic graphs, we now move on to sequences of \textbf{random} dynamical graphs.
 
 \begin{defi}\label{def:weak-loc-conv}
 Given a random sequence $\{ \mathbf G^n,\, n\ge1 \}$ of marked dynamical graphs on $V_n$, we say that
 \begin{enumerate}[label=(\alph*)]
 \item\label{def:conv-loc-weakly-random} $\{ \mathbf G^n, n\ge1\}$ \textbf{converges locally weakly} to $(\mathbf G^\infty,\mathbf o)$ if $\{ (\mathbf G^n,\mathbf o_n), n\ge1 \}$ converges in distribution to $(\mathbf G^\infty,\mathbf o)$. \\
 Equivalently, for every bounded and continuous function $h:\mathcal{DN}_\bullet^S\to\R$:
 $$ \E{h\big( (\mathbf G^n,\mathbf o_n) \big)} = \E{\frac{1}{n} \sum_{v\in V_n} h\big( (\mathbf G^n,v) \big)} \longrightarrow \E{h\big( (\mathbf G^\infty,\mathbf o) \big)} \, . $$
 \item\label{def:conv-prob-loc-random} $\{ \mathbf G^n,\, n\ge1\}$ \textbf{converges locally in probability} to $\{ \mathbf G^n, n\ge1\}$ having (possibly random) distribution $\nu$ if for every bounded and continuous function $h:\mathcal{DN}_\bullet^S\to\R$ we have convergence in probability of
 $$ \E{h\big( (\mathbf G^n,\mathbf o_n) \big) \mid \mathbf G^n} = \frac{1}{n} \sum_{v\in V_n} h\big( (\mathbf G^n_t,v) \big) $$
 to the random variable $\bE_\nu\left[h\big( (\mathbf G^\infty,\mathbf o) \big)\right]$.
 \end{enumerate}
 \end{defi}
 
 Note that in Definition \ref{def:conv-prob-loc-random}, $\bE_\nu\left[h\big( (\mathbf G^\infty,\mathbf o) \big)\right]$ may be a (nondeterministic) random variable due to the possible random choice of $\nu$.
 However, in the common case of a deterministic distribution $\nu$, there is a convenient way of proving local convergence in probability.
 
 \begin{lem}\label{lem:joint-local-cv-to-local-in-proba}
 Let $\nu$ be a probability measure on $\mathcal{DN}_\bullet^S$.
 Given a random sequence $\{ \mathbf G^n, \,n\ge1 \}$ of dynamical graphs on $V_n$, consider $\mathbf o_n$ and $\mathbf o'_n$ two independent uniformly chosen vertices, and suppose that $\left( (\mathbf G^n, \mathbf o_n),(\mathbf G^n, \mathbf o'_n) \right)$ converges in distribution to $\nu\otimes\nu$.
 Then $\{ \mathbf G^n, \,n\ge1 \}$ converges locally in probability to $\nu$.
 \end{lem}
 
 \begin{proof}
 The result follows from a simple second moment computation.
 Let $h$ be continuous and bounded, and $(\mathbf G^\infty_t, \mathbf o)_{t\ge0}$ have distribution $\nu$.
 Then $h\big( (\mathbf G^n_t, \mathbf o_n)_{t\ge0} \big)$ converges in distribution to $h\big( (\mathbf G^\infty_t,\mathbf o)_{t\ge0} \big)$, and as $h$ is bounded, we also get a convergence of the first moment, hence
 $$ \E{\frac{1}{n} \sum_{v\in V_n} h\big( (\mathbf G^n_t,v)_{t\ge0} \big)}
 = \E{h\big( (\mathbf G^n_t, \mathbf o_n)_{t\ge0} \big)}
 \longrightarrow \E{h\big( (\mathbf G^\infty, \mathbf o) \big)} \, . $$
 We now compute the second moment of this random variable as
 \begin{align*}
 \E{\frac{1}{n^2}\sum_{v,v' \in V_n} h\big( (\mathbf G^n_t,v)_{t\ge0} \big) h\big( (\mathbf G^n_t,v')_{t\ge0} \big)}
 &= \E{h\big( (\mathbf G^n_t,\mathbf o_n)_{t\ge0} \big) h\big( (\mathbf G^n_t,\mathbf o'_n)_{t\ge0} \big)} \\
 & \longrightarrow \E{h\big( (\mathbf G^\infty_t, \mathbf o)_{t\ge0} \big)}^2 \, ,
 \end{align*}
 where the convergence comes from the convergence in distribution of $\left((\mathbf G^n_t,\mathbf o_n)_{t\ge0},(\mathbf G^n_t, \mathbf o'_n)_{t\ge0}\right)$ and again the fact that $h$ is continuous and bounded.
 Hence the random variable $\frac{1}{n} \sum_v h\big( (\mathbf G^n_t,v)_{t\ge0} \big)$ converges in $\Lp^2$ and in probability to the constant $\E{h\big( (\mathbf G^\infty_t, \mathbf o)_{t\ge0} \big)}$.
 \end{proof}

\begin{rem} \label{remark_Dynamical_components_or_balls}
	By Definition \ref{def:deterministic-loc-conv}, in order to prove the convergence in distribution of the dynamical components $\left( (\mathbf G^n_t,\mathbf o_n)_{t\ge0},(\mathbf G^n_t,\mathbf o'_n)_{t\ge0} \right)$ to $\nu\otimes\nu$, it suffices to prove that for fixed $d\ge 1$, $T<+\infty$ and $\delta>0$, one can couple the dynamical balls $(\mathbf G^{n,[d]}_t,\mathbf o_n)_{t\le T}$ and $(\mathbf G^{n,[d]}_t,\mathbf o'_n)_{t\le T}$ with $(\mathbf G^{[d]}_t, \mathbf o)_{t\le T}$ and  $(\mathbf G'^{[d]}_t, \mathbf o')_{t\le T}$, where $\mathbf G$ and $\mathbf G'$ are \textit{i.i.d.} dynamical graphs with law $\nu$, so that \emph{w.h.p.} as $n$ tends to infinity, they coincide in $\DGb$ with corresponding marks at distance less than $\delta$ from each other. 
\end{rem}


\section{Growth-and-Segmentation Trees}\label{section:GS-tree}

 In this section we introduce the \textbf{growth-and-segmentation multitype Poisson-Galton-Watson tree}, or more briefly $GSMPGW$ tree, appearing in our main Theorem. It should be seen as an analogue of the (multitype) Galton-Watson tree, in the context of dynamical graphs. In the time-dynamics of this process, the appearance of new edges may make it ``grow'' (in the sense that we concatenate a new tree to the existing one), while the disappearance of edges may split the tree into several connected components, hence obtaining a forest. We still call this process growth-and-segmentation \emph{tree} as it is a dynamically connected graph.
 
 As mentioned earlier, it will be convenient to work with ordered trees. This allows to identify unambiguously any given vertex during the time-dynamics, and to construct the growth-and-segmentation multitype Poisson-Galton-Watson tree as a Markov process on the set of ordered (segmented) trees. The $GSMPGW$ tree can then also be seen as a random element of $\DNb^S$ simply by forgetting the ordering of the children of each vertex, or of $\DGb$ by further forgetting the marks of the vertices.

\subsection{Ordered Multitype Trees and the Multitype Galton-Watson Tree}

 We begin with some notations.
 
 \begin{defi}
 We denote by 
 \begin{enumerate}[label=(\alph*)]
 \item $\T$ the set of \textbf{locally finite ordered rooted} trees. Using Neveu's notations \cite{Neveu}, a tree $\bt\in \T$ is identified with a subset of 
 $$ \mathcal U := \bigcup_{n\ge0} (\N^*)^n \, , $$
 the set of finite words on positive integers, with the convention $(\N^*)^0=\{\emptyset\}$, simply by identifying the $u_n$-th child of the $u_{n-1}$-th child of (...) of the $u_1$-th child of the root, with the word $u=(u_1,\dots,u_n)\in \mathcal U$. In particular, the root is identified with $\emptyset$.
 \item $\ST$ the set of \textbf{segmented} trees, namely trees in $\T$ where some of the edges can be segmented. For $\bt\in \ST$, we write $V_\bt$ for its set of vertices, $\tilde E_\bt$ for its full set of edges, and $E_\bt \subset \tilde E_\bt$ for the unsegmented ones.
 We call $(V_t,\tilde E_t)\in \T$ the full tree (containing both segmented and unsegmented edges) associated to the segmented tree. Note that $\tilde E_\bt \backslash E_\bt$ is then the set of segmented edges, and that the graph $(V_\bt,E_\bt)$ containing only the unsegmented edges is a forest rather than a tree. In the same spirit as percolation, we also say that an edge of $\tilde E_\bt\backslash E_\bt$ is \textbf{closed}, and an edge of $E_\bt$ is \textbf{open}. See Fig. \ref{fig:segmented-tree}.
 \item $\T_S$ the set of \textbf{multitype} trees, or more precisely trees in $\T$ with types in the mark space $S$. Formally, a tree with types in $S$ is a couple $(\bt,m)$ where $\bt$ is in $\T$, and $m$ is a function from the set of vertices of $\bt$ to the mark space $S$, associating its type to each vertex of $\bt$. A multitype tree is called \textbf{well-ordered} if the children of each vertex are ordered in increasing order of their types.
 \item $\ST_S$ the set of \textbf{multitype segmented} trees, or segmented trees with types in $S$.
 \end{enumerate} 
 \end{defi}
 
 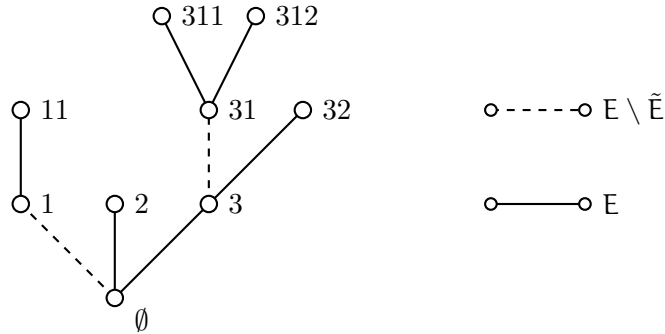
\begin{figure}[!ht]
 \centering
 \begin{tikzpicture}[scale=1.25]
 \draw[thick,dashed] (5,0)--(4,1);
 \draw[thick] (5,0)--(5,1);
 \draw[thick] (5,0)--(6,1);
 \draw[thick] (4,1)--(4,2);
 \draw[thick,dashed] (6,1)--(6,2);
 \draw[thick] (6,1)--(7,2);
 \draw[thick] (6,2)--(5.5,3); \draw[thick] (6,2)--(6.5,3);
 \node[cnvertex] at (5,0) {}; \draw (5.1,0) node[below right] {$\emptyset$};
 \node[cnvertex] at (4,1) {}; \draw (4.1,1) node[right] {$1$};
 \node[cnvertex] at (5,1) {}; \draw (5.1,1) node[right] {$2$};
 \node[cnvertex] at (6,1) {}; \draw (6.1,1) node[right] {$3$};
 \node[cnvertex] at (4,2) {}; \draw (4.1,2) node[right] {$11$};
 \node[cnvertex] at (6,2) {}; \draw (6.1,2) node[right] {$31$};
 \node[cnvertex] at (7,2) {}; \draw (7.1,2) node[right] {$32$};
 \node[cnvertex] at (5.5,3) {}; \draw (5.6,3) node[right] {$311$};
 \node[cnvertex] at (6.5,3) {}; \draw (6.6,3) node[right] {$312$};
 \draw[thick,dashed] (9,2)--(10,2) node[right=2.5pt] {$E\setminus\tilde E$};
 \draw[thick,fill=white] (9,2) circle (1.75pt); \draw[thick,fill=white] (10,2) circle (1.75pt);
 \draw[thick] (9,1)--(10,1) node[right=2.5pt] {$E$};
 \draw[thick,fill=white] (9,1) circle (1.75pt); \draw[thick,fill=white] (10,1) circle (1.75pt);
 \end{tikzpicture}
 \caption{Illustration of an ordered segmented tree. The closed edges are dashed.}\label{fig:segmented-tree}
 \end{figure}
 
 We use the common terminologies for trees, father, children, aso. For a vertex $u$ of a tree $\bt$, we write $|u|$ for its length, which is also the graph distance of the vertex to the root. The \textbf{height} of $\bt$ is defined by
 $$ h(\bt) = \sup\left\{|u|,\, u\in\bt\right\} \, . $$
 For $d\in\N^*$, we denote by $\T^{[d]}$ the subset of $\T$ consisting in trees with height less than $d$. The set $\T^{<\infty}= \cup_d \T^{[d]}$ is then the set of finite-height trees, or trees with finitely many vertices. For a tree $\bt$, we write $\bt^{[d]}$ for its  \textbf{restriction} to the first $d$ generations, or in other words it consists of the ball centred in $\emptyset$ with radius $d$. 
 All these definitions extend straightforwardly to segmented trees, or trees with types in $S$, leading to $\T^{[d]}_S$, $\T^{<\infty}_S$, $\ST^{[d]}_S$ and $\ST^{<\infty}_S$. 
 
 Let $\kappa:S\times S \to \R_+$ be a kernel on the mark space $S$. Then the multitype Poisson-Galton-Watson process with kernel $\kappa$ is classically defined as the branching process, where each particle of type $x\in S$ is replaced in the next generation by a set of individuals distributed as a Poisson point process on $S$ with intensity $\kappa(x,y)\,\mu(\de y)$. When $\kappa$ is a graphical kernel, these Poisson point processes will typically be of finite intensity, and thus define a set of children, which we can order with the natural order on $S$, thus defining an ordered tree.
 
 \begin{defi}
 Let $\kappa$ be a graphical kernel on $S$.
 A random tree on $\T_S$ is called \textbf{multitype Poisson-Galton-Watson tree} with connection kernel $\kappa$ and denoted $MPGW(\kappa)$ if
 \begin{enumerate}[label=(\alph*)]
 \item the root $\emptyset$ has type $x_\emptyset$ distributed according to $\mu$ on $S$, and
 \item a vertex $u$ with type $x_u=x$ has offspring distribution according to a Poisson point process $\Pi^u$ with intensity $\kappa(x,y)\,\mu(\de y)$ on $S$. The children of $u$ are ordered according to their types, thus defining a well-ordered tree.
 \end{enumerate}
 \end{defi}
 
  We also denote by $MPGW(\kappa,x)$ the multitype Poisson-Galton-Watson tree obtained by fixing the type of the root to be $x_\emptyset=x$. Note that the intensity of the Poisson point process $\Pi^u$ is finite for $\mu$-\textit{a.e.} $u$.
 Therefore $MPGW(\kappa)$ is well-defined up to a negligible set, and so is $MPGW(\kappa,x)$ for $\mu$-\textit{a.e.} $x$.  
 The multitype Poisson-Galton-Watson tree is the local weak limit of sparse inhomogeneous random graphs (see for instance \cite[Theorem 3.11]{Hofstad_volume_2}), which is precisely the result we mentioned in the abstract, that we generalize in this paper to dynamical settings.

\subsection{Growing Trees, Segmented Trees and Growth-and-Segmentation Trees}

 In this section, we construct the growth-and-segmentation Poisson-Galton-Watson tree. It should be seen as an analogue of the Poisson-Galton-Watson tree, but in dynamical settings where edges can additionally appear - leading to growth of the tree - or disappear - leading to segmentation of the tree. It will be convenient to first consider only the growth operation and define the growing multitype Poisson-Galton-Watson tree, or $GMPGW$ tree. The $GMPGW$ tree can also be seen as the accumulated graph of the $GSMPGW$ tree that we introduce later on.

\subsubsection{Growing Trees}

 \begin{defi} \label{def:growing_tree}
 A \textbf{growing (multitype) tree} $(\bt_t)_{t\ge0}$ is a growing sequence of ordered rooted multitype trees $\bt_t\in \T_S$, where for $t\ge0$ and vertex $u$ of $\bt_t$, the ordering of the children of $u$ in $\bt_t$ respects the order of appearance of the edges in the process.
 \end{defi}
 
 On $\T_S$ we may define a natural operation which can be used to describe growing multitype trees.
 Let $\br,\bs\in\T_S$ be two trees, and $u$ a vertex of $\br$.
 We denote by $\Gr(\br,\bs,u)$ the tree obtained by adding a child to the vertex $u$ of $\br$ and merging this child with the root of $\bs$.  
 See Fig. \ref{fig:growth-of-tree}. 

 \begin{defi}
 The random growing tree $(\mathfrak G^{\kappa,\beta}_t)_{t\ge0}$ is called \textbf{growing multitype Poisson-Galton-Watson tree} with kernels $\kappa$ and $\beta$ and denoted $GMPGW(\kappa,\beta)$ if it is a Markov process on $\T_S$ where:
 \begin{enumerate}[label=(\alph*)]
 \item\label{def:GMPGW-a} $\mathfrak G^{\kappa,\beta}_0$ is distributed as a $MPGW(\kappa)$, and
 \item\label{def:GMPGW-b} at each vertex $u\in\mathfrak G^{\kappa,\beta}_t$, a $MPGW(\kappa,y)$ is grown at rate $\kappa(x_u,y)\beta(x_u,y) \,\mu(\de y)$.
 \end{enumerate}
 \end{defi}
 
 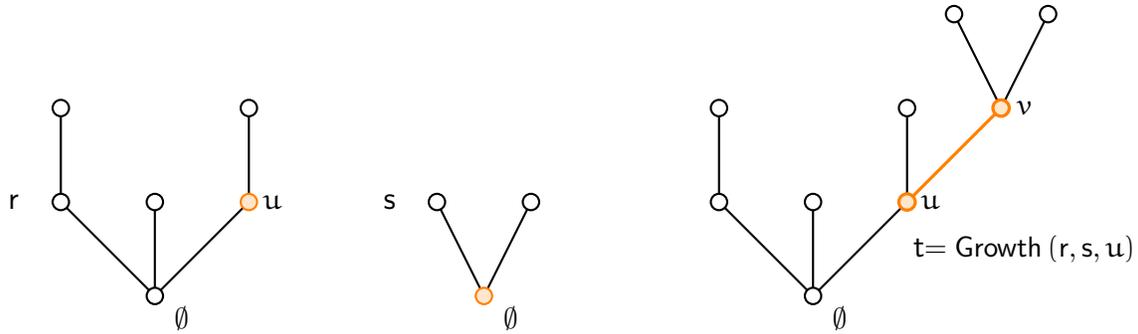
\begin{figure}[!ht]
 \centering
 \begin{tikzpicture}[scale=1.25]
 \draw[thick] (0,0)--(-1,1);
 \draw[thick] (0,0)--(0,1);
 \draw[thick] (0,0)--(1,1);
 \draw[thick] (-1,1)--(-1,2);
 \draw[thick] (1,1)--(1,2);
 \node[cnvertex] at (0,0) {}; \draw (0.1,0) node[below right] {$\emptyset$};
 \node[cnvertex] at (-1,1) {};
 \node[cnvertex] at (0,1) {};
 \node[covertex] at (1,1) {}; \draw (1.25,1) node {$u$};
 \node[cnvertex] at (-1,2) {};
 \node[cnvertex] at (1,2) {};
 \draw (-1.5,1) node {\large{$\br$}};
 
 \draw[thick] (3.5,0)--(3,1);
 \draw[thick] (3.5,0)--(4,1);
 \node[covertex] at (3.5,0) {}; \draw (3.6,0) node[below right] {$\emptyset$};
 \node[cnvertex] at (3,1) {};
 \node[cnvertex] at (4,1) {};
 \draw (2.5,1) node{\large{$\bs$}};
 
 \draw[thick] (7,0)--(6,1);
 \draw[thick] (7,0)--(7,1);
 \draw[thick] (7,0)--(8,1);
 \draw[thick] (6,1)--(6,2);
 \draw[thick] (8,1)--(8,2);
 \draw[very thick,orange] (8,1)--(9,2);
 \draw[thick] (9,2)--(8.5,3);
 \draw[thick] (9,2)--(9.5,3);
 \node[cnvertex] at (7,0) {}; \draw (7.1,0) node[below right] {$\emptyset$};
 \node[cnvertex] at (6,1) {};
 \node[cnvertex] at (7,1) {};
 \node[covertex, very thick] at (8,1) {}; \draw (8.25,1) node {$u$};
 \node[cnvertex] at (6,2) {};
 \node[cnvertex] at (8,2) {};
 \node[covertex, very thick] at (9,2) {}; \draw (9.25,2) node {$v$};
 \node[cnvertex] at (8.5,3) {};
 \node[cnvertex] at (9.5,3) {};
 \draw (9.25,0.5) node {{\large{$\bt$}}$=\Gr\left(\br,\bs,u\right)$};
 \end{tikzpicture}
 \caption{Illustration of the growing operation of trees: the tree $\bt$ (right) is obtained by attaching the tree $\bs$ (center) to the tree $\br$ (left) at vertex $u$ of $\br$.}\label{fig:growth-of-tree}
 \end{figure}
 
 \begin{rem}
 \begin{itemize}
 \item Note that our definition requires both the connection kernel $\kappa$ and the updating kernel $\beta$. Informally, $\kappa(x,y)\,\mu(\de y)$ is the probability of initially finding a child of a vertex of type $x$ which has type in $\de y$. We then multiply this quantity by $\beta(x,y)\,\de t$ to obtain the probability of such an edge being grown in time interval $\de t$.
 \item One can restate \ref{def:GMPGW-b} by saying that a tree is grown at $u$ at total rate
 $$ \int_S \kappa(x_u,z)\beta(x_u,z)\,\mu(\de z) \, , $$
 and is distributed as $MPGW(\kappa,y)$ where $y$ is itself distributed according to the probability measure
 $$ \frac{\kappa(x_u,y)\beta(x_u,y)\,\mu(\de y)}{\int_S \kappa(x_u,z)\beta(x_u,z)\,\mu(\de z)} \,. $$
 Note that the rate $\int \kappa(x_u,z)\beta(x_u,z)\,\mu(\de z)$ is \textit{a.s.} finite for every encountered vertex $x_u$.
 \item 
 The definition might require some caution as the total rate at which new trees are grown to the whole (possibly) infinite tree $\mathfrak G^{\kappa,\beta}_t$ will typically be infinite. One way of making this definition rigorous is by truncating the $GMPGW$ tree to the first $d$ generations to have a well-defined continuous-time Markov chain $(\mathfrak G^{\kappa,\beta,[d]}_t)_{t\ge0}$ on $\T^{[d]}_S$, whose total jumping rate is finite on almost every encountered state in $\T^{[d]}_S$. We can then use the consistent family of the laws of $(\mathfrak G^{\kappa,\beta,[d]}_t)_{t\ge0}$ for $d\ge0$ to define the law of $(\mathfrak G^{\kappa,\beta}_t)_{t\ge0}$, as a Markov process on $\T_S$ with c\`adl\`ag trajectories. 
 
 \item The definition could naturally lead to construct (or explore) the whole progeny of a vertex as soon as it is added to the tree. However, it is not less natural to construct (or explore) the growing set of direct children of a vertex $x$ up to some time $T$ (that may be finite or infinite), before looking at the progeny of these children. When added to the tree, a vertex $u$ then immediately has a progeny given by a Poisson point process of intensity $\kappa(x_u,y)\,\mu(\de y)$, and later on the vertex receives a child of type $y$ at rate $\kappa(x_u,y)\beta(x_u,y)\,\mu(\de y)$. 
 This provides another rigorous construction of the process, naturally equivalent to the previous one.
 \end{itemize}
 \end{rem}

\subsubsection{Growth-and-Segmentation Trees}\label{subsubsection:GST}

 We begin with a definition.
 
 \begin{defi}
 A \textbf{growth-and-segmentation multitype tree} $(\bt_t)_{t\ge0}$ is a sequence of multitype segmented trees $\bt_t\in \ST_S$, where the associated full tree is a growing multitype tree in the sense of Definition~\ref{def:growing_tree}.
 \end{defi}
 
 Our definitions of $\bt=\Gr(\br,\bs,u)$, of $\bt_u$ and $\bt^u$, extend straightforwardly to the case where $\br$, $\bs$ and $\bt$ can now be segmented trees.
 We also denote by $\bt=\Sp(\bs,u)$ the segmented tree obtained by closing the edge between the vertex $u$ of $\bt$ different from the root, and its parent.
 This action defines two sub-segmented trees of $\bt$, one containing the root, denoted $\bt^u$ and called the \textbf{pruned branch} of $\bt$ at $u$, and the second containing $u$, denoted $\bt_u$ and called \textbf{remaining segmented tree} at $u$. See Fig. \ref{fig:removing-edge}.
 
 \begin{figure}[!ht]
 \centering
 \begin{tikzpicture}[scale=1.1]
 \draw[thick,dashed] (0,0)--(-1,1);
 \draw[thick] (0,0)--(0,1);
 \draw[thick] (0,0)--(1,1);
 \draw[thick] (-1,1)--(-1,2);
 \draw[thick] (1,1)--(1,2);
 \draw[thick] (1,1)--(2,2);
 \draw[thick] (1,2)--(0.5,3);
 \draw[thick] (1,2)--(1.5,3);
 \node[cnvertex] at (0,0) {}; \draw(0.25,0) node[right] {$\emptyset$};
 \node[cnvertex] at (-1,1) {};
 \node[cnvertex] at (0,1) {};
 \node[cnvertex] at (1,1) {}; \draw (1.25,1) node[right] {$u$};
 \node[cnvertex] at (-1,2) {};
 \node[cnvertex] at (1,2) {};
 \node[cnvertex] at (2,2) {};
 \node[cnvertex] at (0.5,3) {};
 \node[cnvertex] at (1.5,3) {};
 \draw (-1.5,1) node {$\bs$};
 \draw[thick,dashed] (5,0)--(4,1);
 \draw[thick] (5,0)--(5,1);
 \draw[ultra thick,dashed,orange] (5,0)--(6,1);
 \draw[thick] (4,1)--(4,2);
 \draw[thick] (6,1)--(6,2);
 \draw[thick] (6,1)--(7,2);
 \draw[thick] (6,2)--(5.5,3);
 \draw[thick] (6,2)--(6.5,3);
 \node[cnvertex] at (5,0) {}; \draw (5.25,0) node[right] {$\emptyset$};
 \node[cnvertex] at (4,1) {};
 \node[cnvertex] at (5,1) {};
 \node[cnvertex] at (6,1) {}; \draw (6.25,1) node[right] {$u$};
 \node[cnvertex] at (4,2) {};
 \node[cnvertex] at (6,2) {};
 \node[cnvertex] at (7,2) {};
 \node[cnvertex] at (5.5,3) {};
 \node[cnvertex] at (6.5,3) {};
 \draw (8,0.5) node{{\large{$\bt$}}$=\Sp\left(\bs,u\right)$};
 \end{tikzpicture}
 \caption{Illustration of the segmentation operation: tree $\bs$ (left) is segmented at vertex $u$.}\label{fig:removing-edge}
 \end{figure}
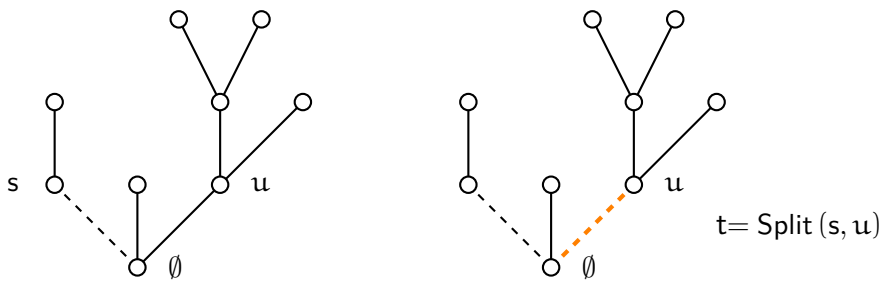

\subsubsection{Growth-and-Segmentation Multitype Poisson-Galton-Watson Tree}

 We can now define the \textit{growth-and-segmentation multitype Poisson-Galton-Watson trees} similarly as the $GMPGW(\kappa,\beta)$, but with the additional dynamics that each edge $uv$ present in the process is segmented at rate $\beta(x_u,x_v)$.
 
 \begin{defi}
 The random growth-and-segmentation tree $(\Fo^{\kappa,\beta}_t)_{t\ge0}$ is called \textbf{growth-and-seg\-mentation multitype Poisson-Galton-Watson tree} with edge kernels $\kappa$ and $\beta$, and denoted $GSMPGW(\kappa,\beta)$ if it is a Markov process on $\ST_S$ where:
 \begin{enumerate}[label=(\alph*)]
 \item $\Fo^{\kappa,\beta}_0$ is distributed as a $MPGW(\kappa)$,
 \item each edge $uv$ present in $\Fo^{\kappa,\beta}_t$ is segmented at rate $\beta(x_u,x_v)$, and
 \item at each vertex $u\in\Fo^{\kappa,\beta}_t$, a $MPGW(\kappa,y)$ is grown at rate $\kappa(x_u,y)\beta(x_u,y)\,\mu(\de y)$.
 \end{enumerate}
 \end{defi}
 
 Similarly as before, this defines strong Markov process on $\ST_S$ with c\`adl\`ag trajectories.

\subsection{Monotone Approximation of $GSMPGW$}

 We finish this part in showing a monotone approximation for $GSMPGW$.
 
 \begin{prop}[Monotone coupling of $GSMPGW$]\label{prop:monotone-coupling-GSMPGW}
 Suppose $\kappa$ and $\beta$ are graphical kernels and $(\kappa_n,\beta_n)_{n\ge0}$ is a sequence of kernels such that $\kappa_n(x,y)\nearrow\kappa(x,y)$ and $\beta_n(x,y)\nearrow\beta(x,y)$.
 Then we can couple the processes $\Fo^{\kappa_n,\beta_n}$ and $\Fo^{\kappa,\beta}$ so that almost surely, for every finite time $T\ge0$ and distance $d\ge1$, we have for all $n$ large enough,
 \begin{equation*} \Fo^{\kappa_n,\beta_n,[d]}_t = \Fo^{\kappa,\beta,[d]}_t \qquad \forall t\le T \, . \end{equation*}
 \end{prop}
 
 Note that it follows from this proposition that $\Fo^{\kappa_n,\beta_n}$ converges in distribution to $\Fo^{\kappa,\beta}$ in $\mathcal{DN}_\bullet$ (or in $\mathcal{DG}_\bullet$) and the result is actually stronger, as in the theorem the marks in $\Fo^{\kappa_n,\beta_n,[d]}_t$ and in $\Fo^{\kappa,\beta,[d]}_t$ are equal when $n$ is large, and not only close to each other.
 
 The proof of this easy result is best seen in the ``temporal first'' construction of the growth-and-segmentation tree.
 In this construction, for each vertex $u$ starting from the root, its children are constructed with the use of
 \begin{itemize}
 \item a Poisson Point Process (PPP) of intensity $\kappa(x_u,y) \mu(\de y)$ as well as another PPP of intensity 
 $$ \kappa(x_u,y) \beta(x_u,y) \mu(\de y) \1{0\le t\le T} \,\de t \, , $$
 in the case of the growth-and-segmentation tree $(\Fo^{\kappa,\beta,[d]}_t)_{0\le t\le T}$,
 \item a PPP of intensity $\kappa_n(x_u,y) \mu(\de y)$ as well as another PPP of intensity 
 $$ \kappa_n(x_u,y) \beta_n(x_u,y) \mu(\de y) \1{0\le t\le T}\, \de t \, , $$ 
 in the case of the growth-and-segmentation tree $(\Fo^{\kappa_n,\beta_n,[d]}_t)_{0\le t\le T}$.
 \end{itemize}
 and we couple the involved PPP so as to guarantee that they \textit{a.s.} coincide for large $n$.
 Hence, almost surely, for $n$ large enough the growing trees coincide up to time $T$ and distance $d$ (including finitely many vertices).
 Finally, taking $n$ large enough also ensures that we can make the segmentation times coincide for all the (finitely many) edges, almost surely. 
 \hfill\qedsymbol


\section{Main Results}\label{section:main-results}

 Our main result shows that for all suitable kernels $\kappa$ and $\beta$, the dynamical inhomogeneous random graph $\big(G^{n,\kappa,\beta}_t\big)_{t\ge0}$ admits a local limit.
 
 \begin{thm}[Local weak convergence in probability
 	]\label{thm:main-result-0}
 As $n$ tends to infinity, the (marked) dynamical inhomogeneous random graph $\big(G^{n,\kappa,\beta}_t\big)_{t\ge0}$ converges locally in probability to (the distribution of) the $GSMPGW(\kappa,\beta)$, in the sense of Definition \ref{def:weak-loc-conv}.
 \end{thm}
 
 To establish this result, from Lemma \ref{lem:joint-local-cv-to-local-in-proba} it suffices to prove that when considering two independent roots uniformly chosen among the vertex set $V_n$ of $G^{n,\kappa,\beta}$, the two dynamical marked graphs converge jointly to two independent copies of the limit, which is the case $k=2$ of the following theorem:
 
 \begin{thm}[Joint local convergence of dynamical components]\label{thm:main-result-1}
 For all $k\in\N$, as $n\to\infty$, the $k$ dynamical components 
 $$ \left( ( G^{n,\kappa,\beta,[\infty]}_t,\mathbf o_1 )_{t\ge 0},\dots,( G^{n,\kappa,\beta,[\infty]}_t,\mathbf o_k )_{t\ge 0} \right) $$
 seen as random variables on $\DNb^S$ and rooted at uniform roots $\mathbf o_1,\ldots,\mathbf o_k$ in $V_n$,
 jointly converge in distribution to $k$ independent copies of the $GSMPGW(\kappa,\beta)$.
 \end{thm}
 
 By Remark \ref{remark_Dynamical_components_or_balls} (or its obvious generalization with $k$ dynamical components), it is actually sufficient, for finite $d$ and $T$ and positive $\delta$, to couple the the dynamical balls up to distance $d$ and time $T$, with corresponding marks at distance less than $\delta$ from each other. This is the content of the following theorem:

 \begin{thm}[Coupling of dynamical balls]\label{thm:coupling-convergence}
 Fix $k\in\N$, $d\in\N$, $\delta>0$ and $T>0$.
 Then one can couple the $k$ randomly rooted dynamical balls 
  $$ \left( ( G^{n,\kappa,\beta,[d]}_t,\mathbf o_1 )_{0\le t\le T},\dots,( G^{n,\kappa,\beta,[d]}_t,\mathbf o_k )_{0\le t\le T} \right) $$
 with $k$ independent copies of the $GSMPGW(\kappa,\beta)$, restricted to distance $d$ and time $[0,T]$, so that they coincide (in $\DG_\bullet$) with corresponding marks at distance less than $\delta$ from each other \textit{w.h.p.} as $n\to\infty$.
 \end{thm}
 
 At this point, we stress that all three theorems above are \emph{annealed}, as the dynamical inhomogeneous random graph is a random model constructed on a possibly random sequence $(x_1,\ldots,x_n)$. However, we directly deduce a similar \textit{quenched} result if we first condition on any realization of the vertex space such that all three convergence results \eqref{eq:nu_n}, \eqref{eq:graphical} and~\eqref{eq:graphical-2} hold (with then a simple limit rather than a limit in probability).
 
 We will prove Theorem \ref{thm:coupling-convergence} first in the context of finite-type kernels, \textit{i.e.} when the state space $S=\{1,\dots,r\}$ is finite, where we will actually obtain a more quantitative quenched result, in which $t$ and $d$ may depend on $n$ but not grow too quickly.
 
 \begin{thm}[Quantitative coupling for finite-type kernels]\label{thm:main-result-2}
 Suppose $S=\{1,\dots,r\}$ is finite, and $k\in \N$, and condition on a realization of the vertex space. Then there exists a finite constant $c=c(\kappa,\beta)$ depending only on the kernels $\kappa$ and $\beta$, such that for any $n\ge 0$ and choices of $t_n$ and $d_n$, we can couple the $k$ randomly rooted dynamical inhomogeneous (marked) random graphs 
 $$ \left( ( G^{n,\kappa,\beta,[d_n]}_t,\mathbf o_1 )_{0\le t\le t_n},\dots,( G^{n,\kappa,\beta,[d_n]}_t,\mathbf o_k )_{0\le t\le t_n} \right) $$
 with $k$ independent copies of $GSMPGW(\kappa,\beta)$, restricted to distance $d_n$ and time $[0,t_n]$, so that they coincide (in $\DNb^S$) with probability at least
  \begin{equation}\label{eq:failure-coupling}
 1- \frac {\left(c(1+t_n)\right)^{2 d_n+1}}{n} - \left(c(1+t_n)\right)^{d_n+1} \sum_{z\in S} \Big|\frac{n_z}{n}-\mu_z\Big| \, .
  \end{equation}
 \end{thm}

Note that this quantitative quenched result implies the annealed result with $t_n=t$ and $d_n=d$ not depending on $n$, simply by observing that the vertex space must have
$$ \E{\sum_{z\in S} \Big|\frac{n_z}{n}-\mu_z\Big|}\to 0 \,, $$
and thus the expectation of~\eqref{eq:failure-coupling} tends to 1 as $n$ tends to infinity.

Note also that Theorem \ref{thm:main-result-2} is of course valid with the choice $t_n=0$, but if you make the natural assumption $t_n\ge 1$ then you can replace the terms $1+t_n$ by $t_n$ in~\eqref{eq:failure-coupling}. A typical application of this theorem is to answer the question of how large you can take $d_n$ and $t_n$ for the coupling to hold with high probability. For example, if the types of the vertices are first chosen randomly and taken i.i.d. with distribution $\mu$, then the term $\sum_{z\in S} \Big|\frac{n_z}{n}-\mu_z\Big|$ is typically of order $1/\sqrt n$. Thus the coupling holds \textit{w.h.p.} at least under the following two choices:
$$ \begin{array}{llll}
d_n=d &\text{fixed } &\text{ and } & t_n=o\left(n^{1/(2d+2)}\right),\ \text{ or}\\
t_n=t &\text{fixed } &\text{ and } & d_n\le \alpha \log n\ \  \text{ with small }\alpha>0 \,.
\end{array} $$

The proof will actually provide an explicit expression for the constant $c$. 
To understand the expression, observe that $\left(c(1+t_n)\right)^{d_n}$ is a typical estimate on the number of vertices contained in the dynamical ball. For each such vertex, we have to couple all the new vertices it connects to up until time $t_n$ for the two processes, which provides the second term. Moreover, for each such two vertices, we have to check that their connection has the same evolution up to time $t_n$ for the two processes (for example of non-existing if the two vertices are in two different balls), which provides the first term.

\subsection*{Sketch of Proofs}\label{subsection:strategy}

Our main results all follow from Theorem~\ref{thm:coupling-convergence} and Theorem~\ref{thm:main-result-2}.

In Section \ref{subsection:from-finite-type-general-kernel}, we prove that Theorem \ref{thm:coupling-convergence} follows from the same result with finite-type kernels, and in particular follows from Theorem \ref{thm:main-result-2}. To this end, we show that the finite graph $(G^{n,\kappa,\beta}_t)_{t\geq0}$ is well-approximated by dynamical inhomogeneous random graphs $(G^{n,\kappa_m,\beta_m}_t)_{t\geq0}$ with ``finite-type'' kernels $\kappa_m$, in the sense that we can couple the two graphs such that the dynamical balls coincide with high probability until fixed time $T\geq0$ and fixed radius $d\geq1$.
Moreover proposition \ref{prop:monotone-coupling-GSMPGW} proves that $GSMPGW(\kappa,\beta)$ is well-approximated by finite-type kernels too.
Combining these two coupling approximation gives us the desired coupling between the two processes with general graphical kernels $\kappa$ and $\beta$.

In Section \ref{subsection:coupling-finitely-many-types}, we finally establish Theorem \ref{thm:main-result-2} with an explicit coupled construction of the dynamical balls for finite $n$ and for the limiting process. The approach is computationally and notationally heavy but based on simple principles:
\begin{itemize}
	\item We introduce a third intermediate process, so as to prove separately a coupling result concerning the time-evolution of the dynamical balls, and one concerning the exploration of a dynamical ball at a fixed time, namely after revealing a new edge of the graph.
	\item For this exploration at a fixed time, we use standard tools (coupling results for binomial and Poisson random variables), already needed when dealing with static graphs.
	\item We finally deal with the time-evolution by showing that the dynamical balls are pure jump Markov processes with transition rates converging fast enough to those of the limiting process.
\end{itemize}


\section{Proofs of Main Results}\label{section:proof}

\subsection{From Finite-Type Kernels to General Graphical Kernels}\label{subsection:from-finite-type-general-kernel}

In this part, we show how we can deduce Theorem \ref{thm:coupling-convergence} from Theorem \ref{thm:main-result-2} (or rather the annealed version of Theorem \ref{thm:main-result-2}) with an approximation argument of general graphical kernels $\kappa$ and $\beta$.

\subsubsection{Regular Finitary Kernels}
 
 Suppose there exists a projection $\pi$ from $S$ to $\tilde S=\{1,\dots,r\}$, such that:
 \begin{itemize}
 \item Each $S_i := \pi^{-1}(i)$ is a $\mu$-continuity set. Then, letting $\tilde \mu=\pi \mu$ and $\tilde x_n = \pi(x_n)$, we obtain the vertex space
 $$ \tilde{\mathcal V} = ( \tilde S,\tilde \mu, (\tilde x^{(n)}_1,\dots,\tilde x^{(n)}_n)_{n\ge1} ) \, . $$
 \item There exist finite-type kernels $\tilde \kappa$ and $\tilde \beta$ on the state space $(\tilde S,\tilde \mu)$ such that for every $x$ and $y$ in $S$, we have
 $$ \kappa(x,y) = \tilde \kappa(\pi(x),\pi(y)),\qquad \beta(x,y) = \tilde\beta(\pi(x),\pi(y)) \, . $$
 \end{itemize}
 In that case, we say $\kappa$ and $\beta$ are \textbf{regular finitary kernels} with associated projection $\pi$ (following the terminology of \cite{Bollobas-Janson-Riordan}).
 
 The projection $\pi$ extends naturally to graphs with marks in $S$, simply by projecting the marks of each vertex.
 This defines in particular a projection from $\ST_S$ to $\ST_{\tilde S}$.
 It is then clear that $\pi(G^{\mathcal V,n,\kappa,\beta})$ and $G^{\tilde{\mathcal V},n,\tilde \kappa,\tilde \beta}$ have the same law, as well as $\pi(\Fo^{\kappa,\beta})$ and $\Fo^{\tilde \kappa,\tilde \beta}$.
 Thus we can apply Theorem \ref{thm:main-result-2}, or rather the annealed version discussed just after this theorem, to obtain the following proposition:
 
 \begin{prop}[Coupling for regular finitary kernels]\label{prop:coupling-regular-finitary-kernels}
 Suppose $\kappa$ and $\beta$ are regular finitary kernels with associated projection $\pi$. Fix $k\in\N$, $d\in\N$ and $T>0$.
 One can couple the $k$ dynamical balls 
 $$ \left( (G^{n,\kappa,\beta,[d]}_t,\mathbf o_n^1)_{0\le t\le T},\dots,(G^{n,\kappa,\beta,[d]}_t,\mathbf o_n^k)_{0\le t\le T} \right) $$
 with $k$ independent copies of the $GSMPGW(\kappa,\beta)$, restricted to distance $d$ and time $[0,T]$, so that their projections coincide in $\DNb^{\tilde S}$ \textit{w.h.p.} as $n\to\infty$.
 \end{prop}
 
 Note here that in the coupling we can only ask the projections of the marks to coincide in $\tilde S$, but we cannot ask the marks themselves to coincide in $S$.
 In particular, on the state space $S$, it may seem that Proposition \ref{prop:coupling-regular-finitary-kernels} implies a convergence in distribution only in $\DGb$ and not in $\DNb^S$.
 
 However, for a given regular finitary kernels, we could modify the associated projection by refining the associated partition of $S$ into the $S_i$, and ask in this refined partition the further request that each $S_i$ be of diameter less than a given $\delta>0$, using the compactness of $S$.
 This would have the effect of providing a coupling where the associated dynamical graphs (asymptotically) not only coincide, but also have the associated marks at distance less than $\delta$.
 As this is true for any $\delta>0$, we indeed can deduce a convergence in distribution in $\DNb^S$.
 We do not detail more this approximation argument, as we will write down in the next subsection a more detailed approximation argument directly for general graphical kernels, and only need Proposition \ref{prop:coupling-regular-finitary-kernels} from the current subsection.

\subsubsection{General Graphical Kernels}

 We now consider $\kappa$ and $\beta$ graphical kernels.
 We consider an approximating partition as in \cite[Lemma 3.6]{Hofstad_volume_2}, which provides approximating kernels $(\kappa_m)_{m\ge1}$ and $(\beta_m)_{m\ge1}$ that satisfy:
 \begin{enumerate}
 \item For each $m$, the kernels $\kappa_m$ and $\beta_m$ are regular finitary.
 \item The sequence of kernels $(\kappa_m)_{m\ge1}$ (resp. $(\beta_m)_{m\ge1}$) is nondecreasing and bounded by $\kappa$ (resp. $\beta$), with for $\mu$-\textit{a.e.} $x$ and $y$ in $S$:
 $$ \kappa_m(x,y) \to \kappa(x,y), \qquad \beta_m(x,y) \to \beta(x,y) \, . $$
 \end{enumerate}
 We then have, as $\kappa$ is a graphical kernel,
 \begin{align*}
 \frac{1}{n} \sum_{u<v} \left(\frac{\kappa_m(x_u,x_v)}{n} \wedge 1\right)
 & \stackrel{(\bP)}{\underset{n,m\to \infty}{\longrightarrow}} 
 \frac{1}{2} \int\int_{S^2} \kappa(x,y)\,\mu(\de x)\mu(\de y) \, , \\
 \frac{1}{n} \sum_{u<v} \left(\frac{\kappa(x_u,x_v)}{n} \wedge 1\right)
 & \stackrel{(\bP)}{\underset{n\to \infty}{\longrightarrow}} 
  \frac{1}{2} \int\int_{S^2} \kappa(x,y)\,\mu(\de x)\mu(\de y) \, ,
 \end{align*}
 and thus also
 $$ \frac{1}{n} \sum_{u<v} \left(\frac{\kappa(x_u,x_v)}{n}\wedge 1\right) - \left(\frac{\kappa_m(x_u,x_v)}{n}\wedge1\right)
 \stackrel{(\bP)}{\underset{n,m\to \infty}{\longrightarrow}} 0 \, . $$
 Similarly,
 $$ \frac{1}{n} \sum_{u<v}^n \beta(x_u,x_v)\left(\frac{\kappa(x_u,x_v)}{n}\wedge1\right)-\beta_m(x_u,x_v)\left(\frac{\kappa_m(x_u,x_v)}{n}\wedge1\right)
  \stackrel{(\bP)}{\underset{n,m\to \infty}{\longrightarrow}} 0 \, . $$
 Using \eqref{eq:approximation-of-kernel}, we deduce that for every $T\ge0$,
 \begin{equation}\label{inequality:sublinear-differing-edges}
 \frac{1}{n} \Prob{\exists t\le T,\, uv\in E(G^{n,\kappa_m,\beta_m}_t)\Delta\, E(G^{n,\kappa,\beta}_t)}
 \underset{n,m\to\infty}{\longrightarrow} 0 \, , 
 \end{equation}
 where $\Delta$ is the notation for the symmetric difference between the two edge sets.
 In other words, when $m$ and $n$ are large, the average number of edges which can differ between the two random graphs within the time interval $[0,T]$, is bounded by a small constant times $n$.
 This result and Proposition \ref{prop:coupling-regular-finitary-kernels} for the kernels $\kappa_m$ and $\beta_m$ are sufficient to prove Theorem \ref{thm:coupling-convergence} for the kernels $\kappa$ and $\beta$, as we now explain.
 
 We fix $k\in\N$, $d\in\N$ and $T>0$.
 Moreover we let $\delta>0$ and $\eps>0$, and aim to show that for large $n$, we can couple the $k$ dynamical balls 
 $$ \left( (G^{n,\kappa,\beta,[d]}_t,\mathbf o_n^1)_{0\le t\le T},\dots,(G^{n,\kappa,\beta,[d]}_t,\mathbf o_n^k)_{0\le t\le T} \right) $$
 with $k$ copies of the $GSMPGW(\kappa,\beta)$, restricted to distance $d$ and time $[0,T]$, so that with probability larger than $1-\eps$, they coincide in $\DGb$, with marks at distance less than $\delta$ from each other.
 
 Choose $K$ large enough so that the probability that the segmented tree $\Fo^{\kappa,\beta,[d]}_T$ has more than $K$ (possibly segmented) edges, is bounded by $\eps$.
 The existence of such a $K$ is guaranteed since $(\Fo^{\kappa,\beta,[d]}_t)_{t\ge0}$ is a non-explosive Markov process on $\ST_S$.
 
 Choose $m$ large enough so that for $n$ larger than $m$, we have
 $$ \frac{1}{n} \sum_{u<v} \Prob{\exists t\le T, \: uv\in E(G^{n,\kappa_m,\beta_m}) \Delta\, E(G^{n,\kappa,\beta})} \le \eps/K \, . $$
 Then as $n\to\infty$, we have \textit{w.h.p.}
 \begin{equation}\label{eq:approximating-number-edges}
 \frac{1}{n} \sum_{u<v} \1{\exists t\le T, \: uv\in E(G^{n,\kappa_m,\beta_m}) \Delta E(G^{n,\kappa,\beta}) } \le 2\eps/K \, .
 \end{equation}
 Now, using Proposition \ref{prop:monotone-coupling-GSMPGW} and increasing the value of $m$ if necessary, we can ensure that with probability at least $1-\eps$, we have
 $$ \Fo^{\kappa_m,\beta_m,[d]}_t = \Fo^{\kappa,\beta,[d]}_t \qquad \forall t\le T \, .$$
 For these regular finitary kernels $\kappa_m$ and $\beta_m$, we can associate a projection $\pi$ on some finite set $\tilde S=\{1,\dots,r\}$ so that $\mathrm{diam}(S_i)\le \delta$ for any $i\le r$.
 Then by Proposition \ref{prop:coupling-regular-finitary-kernels}, we can ensure that for $n$ large enough, we can couple the $k$ dynamical balls 
 $$ \left( (G^{n,\kappa_m,\beta_m,[d]},\mathbf o_n^1)_{0\le t\le T},\dots,(G^{n,\kappa_m,\beta_m,[d]},\mathbf o_n^k))_{0\le t\le T} \right) $$
 with $k$ independent copies of the $GSMPGW(\kappa_m,\beta_m)$, restricted to distance $d$ and time $[0,T]$, so that with error probability less than $\eps$, they coincide in $\DGb$, with marks at distance less than $\delta$ from each other.
 In turn, with error probability less than $2\eps k$, these $k$ independent copies of $GSMPGW(\kappa_m,\beta_m)$ coincide with $k$ independent copies of $GSMPGW(\kappa,\beta)$ when restricted to distance $d$ and time $[0,T]$, and moreover none of the $k$ segmented trees has more than $K$ (possibly segmented) edges at time $T$.
 
 When further Inequality \eqref{eq:approximating-number-edges} is satisfied, the dynamical balls 
 $$ \left( (G^{n,\kappa_m,\beta_m,[d]}_t,\mathbf o_n^1)_{0\le t\le T},\dots,(G^{n,\kappa_m,\beta_m,[d]}_t,\mathbf o_n^k)_{0\le t\le T} \right) $$
 also coincide with the dynamical balls 
 $$ \left( (G^{n,\kappa,\beta,[d]}_t, \mathbf o_n^1)_{0\le t\le T},\dots,(G^{n,\kappa,\beta,[d]}_t,\mathbf o_n^k)_{0\le t\le T} \right) $$
 with error probability bounded by $2\eps k$.
 Finally, for large enough $n$, we accomplish the desired coupling with error probability less than $(2+4k)\eps$.
 \hfill\qedsymbol

\subsection{Coupling of dynamical balls for finite-type kernels}\label{subsection:coupling-finitely-many-types}
	
We finally prove Theorem \ref{thm:main-result-2}, which then implies all our main results. As it is a quenched result, we work conditionally on the vertex space, and thus consider $x_1,\ldots,x_n$ to be deterministic. 
 For $j=1,\dots,k$, let us write $(\mathcal B^{(n),j}_t)_{t\ge0} = ( (V^{(n),j}_t,E^{(n),j}_t), \mathbf o_j )_{t\ge0}$ for the dynamical ball of radius $d_n$ of the dynamical inhomogeneous random graph $(G^{n,\kappa,\beta}_t)_{t\ge0}$ centred at $\mathbf o_j$.
 Let us also write $\hat E^{(n),j}_t$ its respective accumulated edges.
 In the following, we consider
 \begin{align*}
 \Tt_\dagger = \sup\bigg\{ t\ge0,\: & \mathcal B^{(n),1}_t,\dots,\mathcal B^{(n),k}_t \text{ are disjoint segmented trees,} \\
 & \text{and no edge of } \bigcup_j \hat E^{(n),j}_t \text{ has been segmented and after reappeared} \bigg\} \, ,
 \end{align*}
 with the convention $\sup \emptyset = 0$.
 For $t<\Tt$, in the trees $(V^{(n),1}_t,\hat E^{(n),1}_t), \dots, (V^{(n),k}_t,\hat E^{(n),k}_t)$, we can order the children of each vertex by order of appearance, then by the order of their types, and then in case of ties by the natural order on the vertex set $V_n$.
 Thus we can define, for $t<\Tt_\dagger$,
 $$ F^{(n),j}_t = [ V^{(n),j}_t,\hat E^{(n),j}_t,E^{(n),j}_t ] \in \ST^{[d_n]}_S $$
 the ordered segmented tree.
 For a vertex $u\in F^{(n),j}_t\subset \mathcal U$, we write $v_u^j$ the corresponding vertex in $V^{(n),j}_t \subset V_n$, and its type $x_{v_u^j}$ will be simply written $x_u^j$. For $t\ge \Tt_\dagger$, we will simply define $(F^{(n),1}_t,\dots,F^{(n),k}_t) = \dagger$, where $\dagger$ is a cemetery state that we introduce, so that $\Tt_\dagger$ is also the hitting time of this cemetery state. 
We also write $( \Fo^{1,[d_n]}_t,\dots,\Fo^{k,[d_n]}_t)_{t\le t_n}$ for $k$ independent copies of $GSMPGW(\kappa,\beta)$ truncated at height $d_n$. In order to couple $(F^{(n),1}_t,\dots,F^{(n),k}_t)_{t\le t_n}$ with $( \Fo^{1,[d_n]}_t,\dots,\Fo^{k,[d_n]}_t)_{t\le t_n}$, we will actually see that these are Markov processes on $(\ST^{[d_n]}_S )^k \cup \{\dagger\}$ (which might still be unclear for the process $( F^{(n),1}_t,\dots,F^{(n),k}_t )$), whose transition rates are close to another. Note that if the coupling succeeds, then we have $t_n<\Tt_\dagger$, so that we will not be interested in what happens after time $\Tt_\dagger$. \\
 For $t<\Tt_\dagger$, we write $V^{(n),=d_n,j}_t$ for the vertices at distance exactly $d_n$ from the root $\mathbf o_j$ in the tree $(V^{(n),j}_t,\hat E^{(n),j}_t)$, and $V^{(n),[d_n-1],j}_t$ for the remaining vertices at distance at most $d_n-1$.
 An important observation is that every edge $uv$ with $u\in V_n \backslash \bigcup_j V^{(n),[d_n-1],j}_t$ and $v\in V_n\backslash \bigcup_j V^{(n),j}_t$, is open at time $t$ independently with probability $p_{u,v} = \kappa(x_u,x_v) / n$. This is a consequence of the independence of the evolution of these edges from the evolution of the dynamical balls up to time $t$, together with the stationarity property.
 
 Let us consider the different possible transitions modifying the process $( F^{(n),1},\dots,F^{(n),k})$ while it is in some state $(\bfo_1,\ldots,\bfo_k)$.
 \begin{enumerate}
 \item Each edge can be segmented. If $u u'$ is an unsegmented edge in $\bfo_j=F^{(n),j}$, this holds at rate $\beta(x_u^j,x^j_{u'}) (1-\kappa(x^j_u,x^j_{u'})/n)$.
 \item Each segmented edge or each pair of unconnected vertices can get connected. In that case we reach time $\Tt_\dagger$ and by construction the process jumps to the cemetary state $\dagger$. For each such edge between $u\in \bfo_j=F^{(n),j}$ and $u' \in \bfo_{j'}=F^{(n),j'}$, this holds at rate 
 $\beta(x^j_u,x^{j'}_{u'}) \kappa(x^j_u,x^{j'}_{u'}) / n$.
 \item For each $l=1,\dots,k$, for each vertex $u$ of $\bfo_l=F^{(n),[d_n-1],l}_t$ and each type $x\in S$, an edge can appear between the vertex $v_u^l\in V^{(n),[d_n-1],l}_t$ and some vertex $v\in V_n\backslash \bigcup_j V^{(n),l}$ of type $x_v=x$. This holds at rate
 \begin{align*}
 ( n_{x} - | \bigcup_{j=1}^k V^{(n),j}_t(x) | ) & \beta(x^l_u,x) \frac{\kappa(x^l_u,x)}{n} \\
 = & \beta(x^l_u,x) \kappa(x^l_u,x) \frac{n_{x}}{n}
 - \frac{\beta(x^l_u,x) \kappa(x^l_u,x)}{n} \sum_{j=1}^k |V^{(n),j}_t(x)| \, ,
 \end{align*}
 where $V^{(n),j}_t(x)$ is the set of vertices of type $x$ in $F^{(n),j}_t$.
 When this happens, the process $(F^{(n),1},\dots,F^{(n),k})$ will jump to a new random value.
 We will see that its conditional distribution depends only on the current value of the process and on $l$, $u$ and $x$, but not on the values of the vertices $v_u^l$ and $v$ in $V_n$ for example, and we will further write this conditional distribution as $\mu^{(n)}_{(\bfo_1,\dots,\bfo_k),l,u,x}$.
 \end{enumerate}
 To understand this distribution, suppose $( F^{(n),1}_{t-},\dots,F^{(n),k}_{t-} ) = (\bfo_1,\dots,\bfo_k)$ and at time $t$ an edge has been added between $v_u^l\in V^{(n),[d_n-1],l}_{t-}$ and $v\in V_n\backslash \bigcup_j V^{(n),j}_{t-}$ with $x_v=x$. The new value for $( F^{(n),1}_{t},\dots,F^{(n),k}_{t} )$ will then be:
 \begin{itemize}
 \item Either $( \bfo_1,\dots,\Gr(\bfo_l,\bt,u),\dots,\bfo_k )$ for some tree $\bt\in \T^{[d_n-|u|-1]}_S$ with root of type $x$ and height at most $d_n-|u|-1$.
 \item Or the cemetary state if $\mathcal B^{(n),1}_t,\dots,\mathcal B^{(n),k}_t$ are no more disjoint segmented trees.
 \end{itemize}
 Let us fix $\bt\in \T^{[d_n-|u|-1]}_S$ with root of type $x$ and compute the probability of having $( F^{(n),1}_{t},\dots,F^{(n),k}_{t} )=( \bfo_1,\dots,\Gr(\bfo_l,\bt,u),\dots,\bfo_k )$. To this end, we check successively for every $w\in \bt$ in the contour (or lexicographical, or depth-first) order whether the neighbourhood of the corresponding vertex in $V_n$ corresponds to that in $\bt$. This also requires some notation.
 
 Let us write $c_w(\bt)$ for the number of children of $w$ in the tree $\bt$ and $c^y_w(\bt)$ for the number of children of given type $y$. Let us write $e^y_w$ for the number of explored vertices of type $y$ when we explore $w$, including $w$. This is also the number of vertices of type $y$ which arrive no later than $w$ in $\bt$ in the contour order. Let us finally write $s^y_w$ for the number of  ``seen but unexplored'' vertices of type $y$ when we explore $w$. These seen vertices are also the children of the explored vertices, not counting the children of $w$ itself. Let us finally write $u' = u(c_u(\bfo_l)+1)$ so that we will now have $v_{u'}^l = v$. We then have
 \begin{align*}
 \mu^{(n)}_{(\bfo_1,\dots,\bfo_k),l,u,x} & ( \bfo_1,\dots,\Gr(\bfo_l,\bt,u),\dots,\bfo_k ) \\
 = & \prod_{w\in\bt} \prod_{y=1}^r \mubin_{\sum_j |\bfo_j^{=d_n,y}|+s_w^y,\frac{\kappa(x^\bt_w,y)}{n}} (0)
 \prod_{\substack{w\in\bt \\ |w|<d_n-|u|-1}} \prod_{y=1}^r \mubin_{n_y-\sum_j|\bfo_j^y|-s_w^y-e_w^y,\frac{\kappa(x^\bt_w,y)}{n}}(c_w^{y}(\bt)) \,,
 \end{align*}
 where we recall that $\mubin_{n,p}$ denotes the binomial distribution with parameters $n$ and $p$. 
 Indeed, to obtain this result (see Fig. \ref{fig:exploration}), successively for each $w\in\bt$ in the contour order, starting from $w = \emptyset_\bt$ and $v_{u'}^l = v$:
 \begin{itemize}
 \item We first check that the corresponding vertex $v_{u'w}^l$ is not connected to any other vertex in $\bigcup_{j} V^{(n),=d_n,j}_{t-}$ or in the set of seen and unexplored vertices at this time.
 This provides the first term.
 \item If $|w|<d_n-|u|-1$ and thus $|u'w|<d_n$, we check that the number of children of $v_{u'w}^l$ of type $y$ is equal to $c^{y}_w(\bt)$, which provides the second term. Moreover, if this is the case, we define accordingly $v_{u'wi}^l$ for $1\le i\le c^y_w(\bt)$.
 \end{itemize}
 
 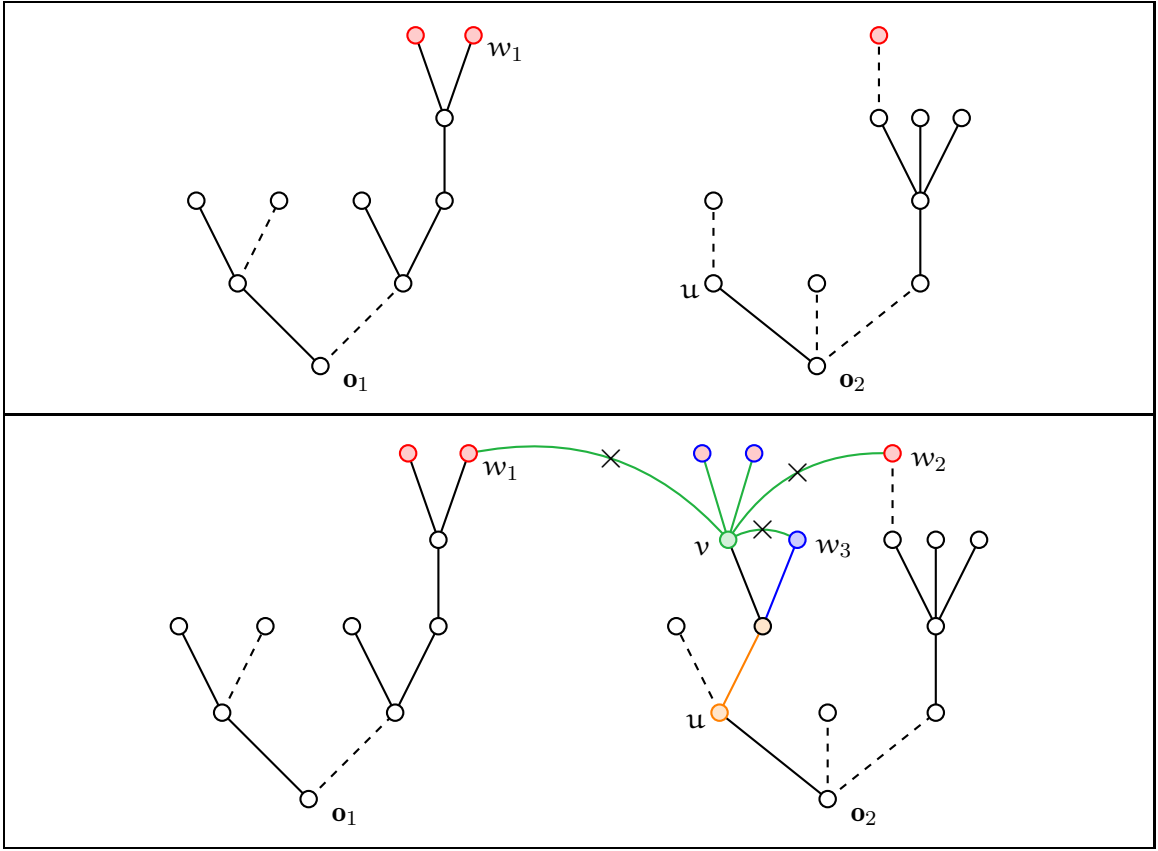
\begin{figure}[!ht]
 \centering
 \begin{tabular}{|>{\centering\arraybackslash}m{0.9\linewidth}|}
 \hline
 \medskip\smallskip
 \begin{tikzpicture}[scale=1.1]
 \draw[thick] (0,0) -- (-1,1);
 \draw[thick,dashed] (0,0) -- (1,1);
 \draw[thick] (-1,1) -- (-1.5,2);
 \draw[thick,dashed] (-1,1) -- (-0.5,2);
 \draw[thick] (1,1) -- (0.5,2);
 \draw[thick] (1,1) -- (1.5,2);
 \draw[thick] (1.5,2) -- (1.5,3);
 \draw[thick] (1.5,3) -- (1.15,4);
 \draw[thick] (1.5,3) -- (1.85,4);

 \node[cnvertex] at (0,0) {}; \draw (0.15,-0.2) node[right] {$\mathbf o_1$};
 \node[cnvertex] at (-1,1) {}; 
 \node[cnvertex] at (1,1) {}; 
 \node[cnvertex] at (-1.5,2) {};
 \node[cnvertex] at (-0.5,2) {};
 \node[cnvertex] at (0.5,2) {}; 
 \node[cnvertex] at (1.5,2) {}; 
 \node[cnvertex] at (1.5,3) {};
 \node[crrvertex] at (1.15,4) {}; 
 \node[crrvertex] at (1.85,4) {}; \draw (1.9,3.8) node[right] {$w_1$};

 \draw[thick] (6,0) -- (4.75,1);
 \draw[thick,dashed] (6,0) -- (6,1);
 \draw[thick,dashed] (6,0) -- (7.25,1);
 \draw[thick,dashed] (4.75,1) -- (4.75,2);
 \draw[thick] (7.25,1) -- (7.25,2);
 \draw[thick] (7.25,2) -- (6.75,3);
 \draw[thick] (7.25,2) -- (7.25,3);
 \draw[thick] (7.25,2) -- (7.75,3);
 \draw[thick,dashed] (6.75,3) -- (6.75,4);
 
 \node[cnvertex] at (6,0) {}; \draw (6.15,-0.2) node[right] {$\mathbf o_2$};
 \node[cnvertex] at (4.75,1) {};  \draw (4.7,0.9) node[left] {$u$};
 \node[cnvertex] at (6,1) {};
 \node[cnvertex] at (7.25,1) {};
 \node[cnvertex] at (4.75,2) {};
 \node[cnvertex] at (7.25,2) {}; 
 \node[cnvertex] at (6.75,3) {};
 \node[cnvertex] at (7.25,3) {}; 
 \node[cnvertex] at (7.75,3) {}; 
 \node[crrvertex] at (6.75,4) {};
 \end{tikzpicture}
 \\
 \hline
\medskip
 \begin{tikzpicture}[scale=1.15]
 \draw[thick] (0,0) -- (-1,1);
 \draw[thick,dashed] (0,0) -- (1,1);
 \draw[thick] (-1,1) -- (-1.5,2);
 \draw[thick,dashed] (-1,1) -- (-0.5,2);
 \draw[thick] (1,1) -- (0.5,2);
 \draw[thick] (1,1) -- (1.5,2);
 \draw[thick] (1.5,2) -- (1.5,3);
 \draw[thick] (1.5,3) -- (1.15,4);
 \draw[thick] (1.5,3) -- (1.85,4);
 
 \draw[thick,Mygreen] (4.85,3) edge [bend left] node[midway,black] {\Large$\times$} (5.65,3);
 \draw[thick,Mygreen] (4.85,3) edge [bend left] node[midway,black] {\Large$\times$} (6.75,4);
  \draw[thick,Mygreen] (4.85,3) edge [bend right] node[midway,black] {\Large$\times$} (1.85,4);

 \node[cnvertex] at (0,0) {}; \draw (0.15,-0.2) node[right] {$\mathbf o_1$};
 \node[cnvertex] at (-1,1) {}; 
 \node[cnvertex] at (1,1) {}; 
 \node[cnvertex] at (-1.5,2) {};
 \node[cnvertex] at (-0.5,2) {};
 \node[cnvertex] at (0.5,2) {}; 
 \node[cnvertex] at (1.5,2) {}; 
 \node[cnvertex] at (1.5,3) {};
 \node[crrvertex] at (1.15,4) {}; 
 \node[crrvertex] at (1.85,4) {}; \draw (1.9,3.8) node[right] {$w_1$};
 
 \draw[thick] (6,0) -- (4.75,1);
 \draw[thick,dashed] (6,0) -- (6,1);
 \draw[thick,dashed] (6,0) -- (7.25,1);
 \draw[thick,dashed] (4.75,1) -- (4.25,2);
 \draw[thick,orange] (4.75,1) -- (5.25,2);
 \draw[thick] (7.25,1) -- (7.25,2);
 \draw[thick] (7.25,2) -- (6.75,3);
 \draw[thick] (7.25,2) -- (7.25,3);
 \draw[thick] (7.25,2) -- (7.75,3);
 \draw[thick,dashed] (6.75,3) -- (6.75,4);
 \draw[thick,black] (5.25,2) -- (4.85,3);
 \draw[thick,blue] (5.25,2) -- (5.65,3);
 \draw[thick,Mygreen] (4.85,3) -- (4.55,4);
 \draw[thick,Mygreen] (4.85,3) -- (5.15,4);
 
 \node[cnvertex] at (6,0) {}; \draw (6.15,-0.2) node[right] {$\mathbf o_2$};
 \node[covertex] at (4.75,1) {};  \draw (4.7,0.9) node[left] {$u$};
 \node[cnvertex] at (6,1) {};
 \node[cnvertex] at (7.25,1) {};
 \node[cnvertex] at (4.25,2) {}; 
 \node[cnvertex] at (7.25,2) {}; 

 \node[cnvertex] at (6.75,3) {};
 \node[cnvertex] at (7.25,3) {}; 
 \node[cnvertex] at (7.75,3) {}; 
 \node[crrvertex] at (6.75,4) {}; \draw (6.85,3.9) node[right] {$w_2$};
 \node[cnovertex] at (5.25,2) {};
 \node[cvvertex] at (4.85,3) {}; \draw (4.75,2.9) node[left] {$v$};
 \node[cbbvertex] at (5.65,3) {}; \draw (5.75,2.9) node[right] {$w_3$};
 \node[cbrvertex] at (4.55,4) {};
 \node[cbrvertex] at (5.15,4) {};
 \end{tikzpicture}
 \\
 \hline
 \end{tabular}
 \caption{Illustration of one step of the exploration process at a growing time. On the top are the two balls just before that growing time. On the bottom a new component is grown at $u$ (orange edge). Blue vertices correspond to stack vertices. Green edges with black cross (not all displayed) are edges checked to be not connected to $v$ when the vertex $v$ (in green) is explored. Black vertices correspond to the explored vertices. Red vertices correspond to those at distance $d$ from its corresponding root. The unseen vertices are not displayed for a better visibility.}\label{fig:exploration}
 \end{figure}
 Finally, the fact that $\mu^{(n)}_{(\bfo_1,\dots,\bfo_k),l,u,x}$ is a probability measure allows to write
 $$ \mu^{(n)}_{(\bfo_1,\dots,\bfo_k),l,u,x}(\dagger) = 1 - \sum_{\bt\in\T^{[d_n-|u|-1]}_S} \mu^{(n)}_{(\bfo_1,\dots,\bfo_k),l,u,x}(f_1,\ldots, \Gr(\bfo_l,\bt,u),\ldots, f_k ) \, . $$
 
 At this point, we have proven that 
 $(F^{(n),1}_t, \ldots, F^{(n),k}_t)$ 
 is indeed markovian and we have provided its transition rates.  
 We will compare these rates with those of $( \Fo^{1,[d_n]},\dots,\Fo^{k,[d_n]} )$, which we can simply make explicit as these are $k$-independent copies of $GSMPGW(\kappa,\beta)$, truncated at height $d_n$. When this process is in state $(\bfo_1,\ldots,\bfo_k)\in\left(\ST^{[d_n]}_S\right)^k$, these rates are given by:
 \begin{enumerate}
 \item Each edge $uu'$ in some $\bfo_j$ is segmented at rate $\beta(x^j_u,x^j_{u'})$.
 \item For each $l=1,\dots,k$, for each vertex $u$ of $\bfo_l=F^{(n),[d_n-1],l}_t$ and each type $x\in S$, a $MPGW(\kappa,x)$ (truncated at height $d_n-|u|-1$) is grown at $u$ at rate $\kappa(x^l_u,x) \beta(x^l_u,x)\mu_x$. Calling  $\mu^{d_n}_{(\bfo_1,\dots,\bfo_k),l,u,x}$ the law of the $k$-tuple of trees after this growth, we thus have for any tree $\bt\in\T^{[d_n-|u|-1]}_S$ with root of type $x$:
 \begin{align*}
 \mu^{d_n}_{(\bfo_1,\dots,\bfo_k),l,u,x}(\bfo_1,\dots,\Gr(\bfo_l,\bt,u),\dots,\bfo_k)
 &= \Prob{MPGW(\kappa,x)^{[d_n-|u|-1]}=\bt}\\
 &= \prod_{\substack{w\in\bt\\|w|<d_n-|u|-1}} \prod_{y=1}^r \mupo_{\kappa(x^\bt_w,y)\mu_y}(c_w^{y}(\bt)),
 \end{align*}
 where we recall that $\mupo_\lambda$ denotes the Poisson law with parameter $\lambda$.
 \end{enumerate}
 
 In order to couple the Markov processes $( F^{(n),1},\dots,F^{(n),k} )$ and $( \Fo^{1,[d_n]},\dots,\Fo^{k,[d_n]} )$, it will be convenient to introduce a third continuous-time Markov process $( \tilde F^{(n),1},\dots,\tilde F^{(n),k} )$, with initial condition $( \tilde F^{(n),1}_0,\dots,\tilde F^{(n),k}_0 ) = ( \Fo^{1,[d_n]}_0,\dots,\Fo^{k,[d_n]}_0 )$, and with the same jump rates as those of $( F^{(n),1},\dots,F^{(n),k} )$, except that when a new vertex of type $x$ is added to one of the segmented tree $\tilde F^{(n),l}$, we use the distribution $\mu^{d_n}_{(\bfo_1,\dots,\bfo_k),l,u,x}$ instead of $\mu_{(\bfo_1,\dots,\bfo_k),l,u,x}^{(n)}$ to determine the state of the process.
 To this end, we will show that it is unlikely that the segmented trees contain too many vertices.
 
 For given $(\bfo_1,\dots,\bfo_k)$, the sum over $(\bfo_1',\dots,\bfo_k')$ of the absolute value of the differences between the transition rates from $(\bfo_1,\dots,\bfo_k)$ to $(\bfo_1',\dots,\bfo_k')$ for the Markov chains $(\Fo^{1,[d_n]},\dots,\Fo^{k,[d_n]})$ and $(\tilde F^{(n),1},\dots,\tilde F^{(n),k})$ is bounded by
 \begin{align*} \max_{x,y\in S}\beta(x,y)\kappa(x,y) \Big( \frac{\mathbf f(\mathbf f-1)}{2n} &+ \frac{\mathbf f^2}{n} +\mathbf f \sum_{z\in S}\left|\frac{n_z}{n}-\mu(z)\right| \Big) \\
 & \le \max_{x,y \in S} \kappa(x,y)\beta(x,y) \left(\frac {3\mathbf f^2}{2n}+\mathbf f \sum_{z\in S} \Big|\frac{n_z}{n}-\mu_z\Big|\right) \, ,
 \end{align*}
 where $\mathbf f = |\bfo_1|+\dots+\bfo_k|$. \\
 The first term with 
 $\mathbf f(\mathbf f-1)/2n$ comes from comparing, for each unordered pair of vertices $u$ and $u'$ of $\bfo_1,\ldots, \bfo_k$, the rate at which the edge is segmented if $uu'$ is the edge of some $\bfo_j$, or the rate at which it appears otherwise (in $V_n$, leading to the process $(\tilde F^{(n),1},\dots,\tilde F^{(n),k})$ jump to the cemetery state). The second and third term come from comparing the rates of growth.

 Therefore, from Lemma \ref{lem:coupling-Markov-process} we can couple the two processes so that
 \begin{align}\label{eq:couplage1}
 \bP\Big( \exists t\le t_n, & \big(\tilde F^{(n),1}_t,\dots,\tilde F^{(n),k}_t\big) \neq \big(\Fo^{1,[d_n]}_t,\dots,\Fo^{k,[d_n]}_t\big) \Big) \nonumber \\
 &\le t_n \max_{x,y}\kappa(x,y)\beta(x,y) \bigg( \frac{3\E{\mathbf Z_n^2}}{2n} + \E{\mathbf Z_n} \sum_{z\in S}\Big|\frac{n_z}{n}-\mu_z\Big| \bigg) \, , 
 \end{align}
where
\begin{equation}
\mathbf Z_n := \left|\Fo^{1,[d_n]}_{t_n}\right|+ \ldots+  \left|\Fo^{k,[d_n]}_{t_n}\right|.
\end{equation}

 It remains to couple the two processes $(F^{(n),1},\dots,F^{(n),k})$ and $(\tilde F^{(n),1},\dots,\tilde F^{(n),k})$ up to time $t_n$.
 To this end, we first observe that the $k$ roots of $(F^{(n),1},\dots,F^{(n),k})$ correspond to distinct vertices in $V_n$ and have types coupled with those of  $(\tilde F^{(n),1},\dots,\tilde F^{(n),k})$ with failure probability roughly bounded by
 $$ k\left(\frac k n +\sum_{z\in S} \Big|\frac{n_z}{n}-\mu_z\Big| \right) \,. $$

 Next, since two edges do not flip simultaneously, it suffices to couple the growing operations encoded with the distributions $\mu^{(n)}_{(\bfo_1,\dots,\bfo_k),l,u,x}$ and $\mu^{d_n}_{(\bfo_1,\dots,\bfo_k),l,u,x}$.
 We thus couple successively the offspring of each encountered vertex, that we have ordered according to their order of appearance, then (in case of ties) according to the index of the tree between $1$ and $k$, and finally, for each grown tree, according to the contour order.\\
 Suppose the $j-$th (appeared) vertex is $u$ in $F^{(n),l}$ (and in $\tilde F^{(n),l}$ as well), and the number of revealed vertices before we reveal the offspring of $u$ is $n(j)$.
 When exploring its offspring, we have to consider:
 \begin{itemize}
 \item A family of Poisson random variables $(P^u_y)_{y\in S}$ with parameters $\kappa(x_u,y) \mu_y$ yielding its number of children of type $y$, in the case of $\tilde F^{(n),l}$ (and when $u$ is not already at maximal depth $d_n$).
 \item A family of binomial random variables $(B^u_y)_{y\in S}$ and $(B'^u_y)_{y\in S}$ with parameters $(b_y^u,\kappa(x_u,y)/n)$ and $(n_y-b'^u_y,\kappa(x_u,y)/n)$ for some $b_y^u$ and $b'^u_y$ satisfying $\sum_y b_y^u\le n(j)$ and $\sum_y b'^u_y\le n(j)$, and yielding:
 \begin{itemize}
 \item the number of connections to previously explored and revealed vertices of type $y$ in $(F^{(n),1},\dots,F^{(n),k})$,
 \item the number of children of type $y$ (when $u$ not already at maximal depth $d_n$)
 \end{itemize}
 in the case of $F^{(n),l}$.
 \end{itemize}
 The total variation distance between a binomial $\mubin_{n,p}$ distribution and a Poisson $\mupo_{np}$ distribution with the same mean is at most $p$ (see for instance \cite{Barbour-Holst-Janson}). Thus all the $B_y^u$ are $0$ with failure probability at most $\sum_y \kappa(x_u,y) b^u_y/n$, and we can couple each $B'^u_y$ with $P^u_y$ with failure probability at most
 \begin{align*}
 & d_{TV}( \mubin_{n_y-b'^v_y,\kappa(x_v,y)/n},\mupo_{\kappa(x_v,y)\mu_y} ) \\
 & \le d_{TV}( \mubin_{n_y-b'^v_y,\kappa(x_v,y)/n},\mubin_{n_y,\kappa(x_v,y)/n)} ) + d_{TV}( \mubin_{n_y,\kappa(x_v,y)/n},\mupo_{\kappa(x_v,y)\mu_y} ) \\
 & \le  \kappa(x_v,y) \frac{b'^v_y}{n} + d_{TV}( \mubin_{n_y,\kappa(x_v,y)/n},\mupo_{\kappa(x_v,y)n_y/n} ) + d_{TV}( \mupo_{\kappa(x_v,y) n_y/n},\mupo_{\kappa(x_v,y)\mu_y} ) \\
 & \le \kappa(x_v,y) \Big( \frac{ b'^v_y+1}{n} + \left| \frac{n_y}{n} - \mu_y \right| \Big) \, .
 \end{align*}
 Hence, all these couplings succeed for vertex $u$ with failure probability bounded by 
 $$ \max_{x,y \in S} \kappa(x,y) \left(\frac {2 n(j) + |S|}n+ \sum_{z\in S} \Big|\frac{n_z}{n}-\mu_z\Big|\right) \,. $$
 We repeat this for all the vertices encountered until time $t_n$. 
 On the event that the first coupling succeeded, the number of these encountered vertices is
$$ \left|\tilde F^{(n),1}_{t_n}\right|+ \ldots+  \left|\tilde F^{(n),k}_{t_n}\right|=   \left|\Fo^{1,[d_n]}_{t_n}\right|+ \ldots+  \left|\Fo^{k,[d_n]}_{t_n}\right| = \mathbf Z_n \,. $$
 For $j\le \mathbf Z_n$ we also have $n(j)\le \mathbf Z_n$, and therefore this second coupling fails with  probability bounded by
 \begin{equation}\label{eq:couplage2}
 \left(1+\max_{x,y \in S}\kappa(x,y) \right) \left(\frac {2\,\E{ \mathbf Z_n^2} + |S|\, \E{\mathbf Z_n}}n+ \E{\mathbf Z_n} \sum_{z\in S} \Big|\frac{n_z}{n}-\mu_z\Big|\right) \, , 
 \end{equation}
 and in particular the probability of $ T_\dagger<t_n$ is bounded by the same quantity.
 
 Combining \eqref{eq:couplage1} and \eqref{eq:couplage2}, we finally obtain that we can couple the three processes $(F^{(n),1}_t,\dots,F^{(n),k}_t)$, $(\tilde F^{(n),1}_t,\dots,\tilde F^{(n),k}_t)$ and $(\Fo^{1,[d_n]}_t,\dots,\Fo^{k,[d_n]}_t)$ up until time $t_n$ and with failure probability bounded by:
 \begin{equation}\label{eq:explicit-failure-coupling}
 \left(1+\max_{x,y \in S}\kappa(x,y)( 1+\beta(x,y)t_n ) \right) \left(\frac {2\,\E{ \mathbf Z_n^2} + |S|\, \E{\mathbf Z_n}}n+ \E{\mathbf Z_n} \sum_{z\in S} \Big|\frac{n_z}{n}-\mu_z\Big|\right) \,.
 \end{equation}
 To conclude Theorem \ref{thm:main-result-2}, it remains to prove that the first two moments of $\mathbf Z_n$ can be bounded by $\left(c(1+t_n)\right)^{d_n}$ and $\left(c(1+t_n)\right)^{2d_n}$ respectively, for some constant $c>0$. It clearly suffices to prove this when there is only one ball, so we suppose $k=1$.
 Then we can roughly bound $\mathbf Z_n$ by the number $\tilde {\mathbf Z}_n$ of vertices at distance at most $d_n$ from the root of a $PGW$ tree with parameter $\lambda_n=1+\max_{x,y\in S} \kappa(x,y) (1+\beta(x,y) t_n)$, and we can assume $\lambda_n\ge 2$ by increasing its value if necessary. 
 From a classic analysis of Galton-Watson processes (see for instance \cite{Harris-branching1,Harris-branching2,Athreya-Ney}), we get the following estimates
 \begin{align*}
\lambda_n^{d_n}&\le \E{\tilde {\mathbf Z}_n} = \frac{\lambda_n^{d_n+1}-1}{\lambda_n-1}\le \frac {\lambda_n}{\lambda_n-1} \lambda_n^{d_n}\le 2 \lambda_n^{d_n},\\
 \lambda_n^{2 d_n}&\le \E{(\tilde {\mathbf Z}_n)^2} \le 2 \left(\frac {\lambda_n}{\lambda_n-1}\right)^3 \lambda_n^{2 d_n}\le 16 \lambda_n^{2 d_n},
 \end{align*}
and this concludes the proof.
 \hfill\qedsymbol \\
 
 As final remarks, we stress that  we clearly could have provided better estimates of the moments of $\mathbf Z_n$, for example as most vertices in the dynamical balls are added at times close to $t_n$ and have less time to generate offspring than the initial root vertices. We also could have obtained the same formula as~\eqref{eq:explicit-failure-coupling} with $\E{\mathbf Z_n\wedge \omega_n}$ instead of $\E{\mathbf Z_n}$, with arbitrary choice of $\omega_n$, simply by stopping the process when $\mathbf Z_n$ exceeds $\omega_n$, but at the cost of adding an error term $\Prob{\mathbf Z_n>\omega_n}$.


\section{Vertex Updating Inhomogeneous Random Graph}\label{section:vertex-updating}

 In this section, we discuss briefly \textbf{vertex updating inhomogeneous random graphs}, a model of dynamical graph with a different kind of dynamics where several edges can be updated simultaneously. 
 More precisely, we consider as before a vertex space $\mathcal V = (S,\mu,(x_1^{(n)},\dots,x_n^{(n)})_{n\ge1})$ and a \textbf{connection kernel} $\kappa:S\times S\to [0,\infty)$, but we now consider an \textbf{updating function} $\beta:S\to[0,\infty)$ which is a function of only one variable. 
 The dynamic is then the following: all edges adjacent to a vertex are updated simultaneously at rate depending on the type of the vertex. Upon updating vertex looses all its connections, and new connections are built independently. See Fig. \ref{fig:vertex_updating}.\\
 
 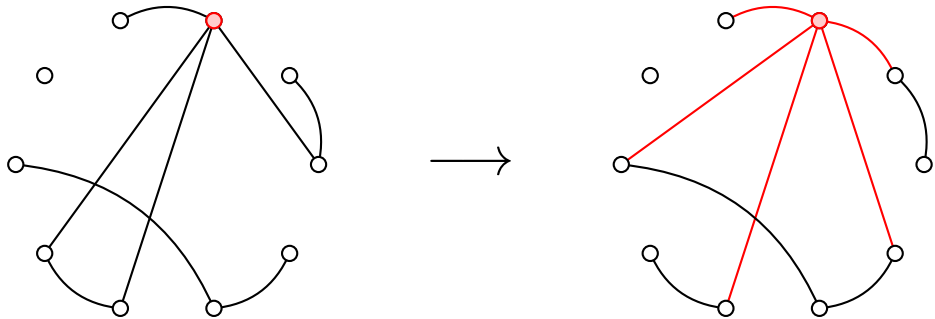
\begin{figure}[!ht]
 \centering
 \begin{tikzpicture}[scale=1.15]
 \path[thick] ({360/10 * (1 - 1)}:1.75) edge [bend right] ({360/10 * (2 - 1)}:1.75);
 \draw[thick] ({360/10 * (1 - 1)}:1.75) -- ({360/10 * (3 - 1)}:1.75);
 \draw[thick] ({360/10 * (3 - 1)}:1.75) edge [bend right] ({360/10 * (4 - 1)}:1.75);
 \draw[thick] ({360/10 * (3 - 1)}:1.75) -- ({360/10 * (7 - 1)}:1.75);
 \draw[thick] ({360/10 * (3 - 1)}:1.75) -- ({360/10 * (8 - 1)}:1.75);
 \path[thick] ({360/10 * (6 - 1)}:1.75) edge [bend left] ({360/10 * (9 - 1)}:1.75);
 \path[thick] ({360/10 * (7 - 1)}:1.75) edge [bend right] ({360/10 * (8 - 1)}:1.75);
 \path[thick] ({360/10 * (9 - 1)}:1.75) edge [bend right] ({360/10 * (10 - 1)}:1.75);
 \foreach \s in {1,...,10}
 { \draw [black, thick, fill=white] ({360/10 * (\s - 1)}:1.75) circle (2.5pt); }
 \draw [red, thick, fill=red!20] ({360/10 * (3-1)}:1.75) circle (2.5pt);
 
 \node [] at (3.5,0) {\huge $\longrightarrow$};
 
 \begin{scope}[xshift=7cm]
 \path[thick] ({360/10 * (1 - 1)}:1.75) edge [bend right] ({360/10 * (2 - 1)}:1.75);
 \path[thick,red] ({360/10 * (3 - 1)}:1.75) edge [bend left] ({360/10 * (2 - 1)}:1.75);
 \draw[thick,red] ({360/10 * (3 - 1)}:1.75) edge [bend right] ({360/10 * (4 - 1)}:1.75);
 \draw[thick,red] ({360/10 * (3 - 1)}:1.75) -- ({360/10 * (6 - 1)}:1.75);
 \draw[thick,red] ({360/10 * (3 - 1)}:1.75) -- ({360/10 * (8 - 1)}:1.75);
 \draw[thick,red] ({360/10 * (3 - 1)}:1.75) -- ({360/10 * (10 - 1)}:1.75);
 \path[thick] ({360/10 * (6 - 1)}:1.75) edge [bend left] ({360/10 * (9 - 1)}:1.75);
 \path[thick] ({360/10 * (7 - 1)}:1.75) edge [bend right] ({360/10 * (8 - 1)}:1.75);
 \path[thick] ({360/10 * (9 - 1)}:1.75) edge [bend right] ({360/10 * (10 - 1)}:1.75);
 \foreach \s in {1,...,10}
 { \draw [black, thick, fill=white] ({360/10 * (\s - 1)}:1.75) circle (2.5pt); }
 \draw [red, thick, fill=red!20] ({360/10 * (3-1)}:1.75) circle (2.5pt);
 \end{scope}
 \end{tikzpicture}
 \caption{One move of the dynamic of the dynamical inhomogeneous random graph with vertex updating. Red vertex on the left is the one chosen to be updated. The red edges on the right are the newly formed edges.}\label{fig:vertex_updating}
 \end{figure}
 
  More formally, given the sequence $x_1,\dots,x_n$, $\big(H^{n,\kappa,\beta}_t\big)_{t\ge0}$ is the random dynamical graph on $V_n$, such that any two vertices $u$ and $v$ are initially connected by an edge independently of the others and with a probability
 $$ p_{u,v} = \frac{1}{n} \kappa(x_u,x_v) \wedge 1, $$
 and each vertex $u$ updates independently with rate $\beta_{u} = \beta(x_u)$.
 That is, after an exponential time with expectation $1/\beta_{u}$, every unordered pair $\{u,v\}$, for $u\neq v$ forms an edge with probability $p_{u,v}$ independently of its previous state and all other edges. The remaining edges $vw$ with $v,w\neq u$ remain unchanged.
 
 We obtain the local limit of these dynamical graphs in a similar manner as for the dynamical inhomogeneous random graphs. 
 In order to describe the limiting model, we need the growing operation $\Gr(\br,\bs,u)$ which was introduced in Section~\ref{subsubsection:GST}, as well as a new operation, which we write $\Fu(\br,\bs,u)$, and which is the tree obtained by merging the root of the tree $\bs$ with the vertex $u$ of $r$, assuming this root and $u$ have the same type.
 
 \begin{defi}
 	The random growth-and-segmentation tree $(\Ho^{\kappa,\beta}_t)_{t\ge0}$ is called \textbf{vertex-growth-and-segmentation multitype Poisson-Galton-Watson tree} with edge and vertex kernels $\kappa$ and $\beta$, and denoted by $VGSMPGW(\kappa,\beta)$ if it is a Markov process on $\ST_S$ where:
 	\begin{enumerate}
 		\item $\Ho^{\kappa,\beta}_0$ is distributed as a $MPGW(\kappa)$,
 		\item for each vertex $u\in\Ho^{\kappa,\beta}_t$, at rate $\beta(x_u)$, all its incident edges are segmented, and a new $MPGW(\kappa,x_u)$ is merged at $u$,
 		\item at each vertex $u\in\Ho^{\kappa,\beta}_t$, a $MPGW(\kappa,y)$ is grown at rate $\kappa(x_u,y)\beta(y)\,\mu(\de y)$.
 	\end{enumerate}
 \end{defi}
 Similarly as before, this defines a strong Markov process on $\ST_S$ with c\`adl\`ag trajectories. The reason why we now have two different growing operations, is that the dynamical ball can grow either by the update of one of its own vertices, leading to a merging operation, or by the update of one external vertex that then connects to one vertex of the ball, leading to a growing operation.
 
 All the results of Section \ref{section:main-results} can be adapted to these new settings, in a similar way that we do not detail. We just state the rephrasing of the first and last result in these new settings.

 \begin{thm}[Local weak convergence in probability]\label{thm:main-result-0-vertex}
 	As $n$ tends to infinity, the (marked) vertex updating inhomogeneous random graph $\big(H^{n,\kappa,\beta}_t\big)_{t\ge0}$ converges locally in probability to (the distribution of) the $VGSMPGW(\kappa,\beta)$.
 \end{thm}

 \begin{thm}[Coupling for the finite-type kernels]\label{thm:main-result-2-vertex}
 	Suppose $S=\{1,\dots,r\}$ is finite, and $k\in \N$. Then there exists a finite constant $c=c(\kappa,\beta)$ depending only on the kernels $\kappa$ and $\beta$, such that for any $n\ge 0$ and choices of $t_n$ and $d_n$, we can couple the $k$ randomly rooted dynamical balls
 	$$ \left( ( H^{n,\kappa,\beta,[d_n]}_t,\mathbf o_1 )_{0\le t\le t_n},\dots,( H^{n,\kappa,\beta,[d_n]}_t,\mathbf o_k )_{0\le t\le t_n} \right) $$
 	with $k$ independent copies of $VGSMPGW(\kappa,\beta)$, restricted to distance $d_n$ and time $[0,t_n]$, so that they coincide (in $\DNb^S$) with probability at least
 	\begin{equation}\label{eq:failure-coupling-vertex}
 	1- \frac {\left(c(1+t_n)\right)^{2 d_n+1}}{n} - \left(c(1+t_n)\right)^{d_n+1} \sum_{z\in S} \Big|\frac{n_z}{n}-\mu_z\Big|.
 	\end{equation}
 \end{thm}


\section{Application to the Contact Process}\label{section:CP}

We briefly describe an application of the theory of local convergence for dynamical graphs developed in this paper. In a recent work \cite{Jacob-Linker-Morters22}, the authors consider the contact process on a dynamical inhomogeneous random graph.
 Their graph model fits our settings when we take the mark space to be $S=[0,1]$, $x_i^{(n)}=i/n$, with connection kernel to be either of the following kernels:
 \begin{itemize}
 \item the \emph{factor kernel} $\kappa(x,y)= a x^{-\gamma} y^{-\gamma}$,
 \item the \emph{preferential attachment kernel} $\kappa(x,y)=a (x \wedge y)^{-\gamma} (x \vee y)^{\gamma-1}$, 
 \item the \emph{strong  kernel} $\kappa(x,y)=a (x \wedge y)^{-\gamma}$, 
 \item the \emph{weak kernel} $\kappa(x,y)=a (x \vee y)^{-\gamma-1}$,
 \end{itemize}
 and updating kernel
 $$ \beta(x,y)=b \left(x^{-\gamma \eta}+y^{-\gamma \eta}\right) \,, $$
 where $a>0$, $b>0$, $\gamma\in (0,1)$ and $\eta\in \R$ are various parameters of the model.
 Applying Theorem \ref{thm:main-result-0}, we directly obtain that their randomly rooted graphs converge locally in probability to the corresponding $GSMPGW(\kappa,\beta)$.
 
 Still in~\cite{Jacob-Linker-Morters22}, on top of this graph dynamics, the authors consider the contact process, which is a popular model of spread of an infection in a population, defined as a continuous-time Markov chain, where every infected vertex infects each of its neighbours with rate $\lambda$ and becomes healthy with rate $1$. It is easily seen that on finite graphs, the infection will become extinct in finite time. 
However, they are interested in the size of the extinction time of the infection, and they showed that there is a transition phase (in the parameters $\gamma$ and $\eta$) between \textit{fast extinction} and \textit{slow extinction}. 
We say that fast extinction occurs if, for some sufficiently small infection rate $\lambda>0$, the expected extinction time is bounded by a power of the network size. 
We say that there is slow extinction if, for every infection rate $\lambda>0$, the mean extinction time is exponential in the number of vertices. 
In the slow extinction regime, they could obtain the following \textbf{metastability result}: 
writing $I_n(t)$ for the expected density of infected vertices at time $t$ in the contact process run on the dynamical graphs with initially all $n$ vertices infected, there exists some $\varepsilon>0$ such that:
\begin{enumerate}
	\item\label{metastable 1} For every sequence $(t_n)$ going to infinity slower than $e^{\varepsilon n}$, we have 
	$$ \liminf_{n\to\infty} I_n(t_n) > 0 . $$
	\item\label{metastable 2} For every sequences $(s_n)$ and $(t_n)$ both going to infinity slower than $e^{\eps n}$, we have 
	$$ I_n(s_n)-I_n(t_n) \underset{n\to \infty} \longrightarrow 0 \, . $$
 \end{enumerate}

 Consequently, we can unambiguously define the lower metastable density $\rho^-(\lambda)=\liminf_{n\to\infty} I_n(t_n)>0$ and the upper metastable density $\rho^+(\lambda)=\limsup_{n\to\infty} I_n(t_n)$, whenever $(t_n)$ goes to infinity slower than $e^{\varepsilon n}$.
 Take any sequence $(r_n)$ going to infinity slower than $e^{\varepsilon n}$ and set $a_n := I_n(r_n)$.  The function $I_n$ is actually non-increasing, so in particular whenever $1\ll s_n\ll t_n \le e^{\varepsilon n}$, the function $I_n$ is approximately constant equal to $a_n$ on the whole time-interval $[s_n,t_n]$ from \ref{metastable 2}. And from \ref{metastable 1}, $\liminf a_n>0$ implies that $a_n>c>0$ for some constant $c>0$ and all $n\ge0$ sufficiently large. 
 However, they could not show how this value depends on $n$, and in particular, whether $\rho^+(\lambda)=\rho^-(\lambda)$ or not. 
 This identity is however a simple consequence of our local convergence result, as we now explain.

 \begin{thm}
 	[The metastable density in the slow extinction regime]
 	\label{thm:metastable_density}
 	Write $\theta(\lambda)$ for the probability that the contact process with infection rate $\lambda$ survives on the $GSMPGW(\kappa,\beta)$, starting from only the root infected. Then, in the slow extinction regime, we have for every $\lambda>0$:
 	\[
 	\rho^+(\lambda)=\rho^-(\lambda)=\theta(\lambda)>0 \,.
 	\]
 \end{thm}

 \begin{proof}[Proof of Theorem \ref{thm:metastable_density}]
 Using the duality property of the contact process, see for instance \cite[Proposition 3]{Jacob-Linker-Morters19}, and of the dynamical inhomogeneous graph, we can rewrite $I_n(t)$ as
$$ I_n(t)=\frac 1 n \sum_{v\in \{1,\ldots, n\}} \bP_v\left(\mbox{Infection has not been extinct by time }t\right) \,, $$
where $\bP_v$ stands for the contact process starting from only $x$ infected.
 Looking again at the proof of metastability by \cite{Jacob-Linker-Morters22}, the authors actually prove that for every small error $r>0$, there exists $t>0$ such that uniformly in large $n$, we have $I_n(t_n)- I_n(t)\le r$. Therefore, it suffices to prove the convergence of $I_n(t)$ when $n\to 0$ for fixed $t$ to obtain the claimed result. To this end, simply consider the functional $h_t:\DGb\to [0,1],$ which to a given dynamical graph, associates the probability that a contact process started from only the root infected, survives in this dynamical graph up to time $t$. The functional $h_t$ is clearly continuous. Hence we obtain
 \begin{equation} \label{convergence_Int}
 I_n(t)  \underset{n\to \infty} \longrightarrow 
 \E{h_t(GSMPGW(\kappa,\beta))} \, ,
 \end{equation}
 and subsequently $\rho^+(\lambda)=\rho^-(\lambda)$ as claimed. We can further write it as simply $\rho(\lambda)$ and call it the \emph{metastable density}. Furthermore, using that the graphs converge \emph{locally in probability} and not only \emph{locally weakly}, we can strengthen~\eqref{convergence_Int} by conditioning on the evolution of the dynamical graph (but not of the contact process run on this graph) and obtain the convergence in probability of the corresponding random variables. Finally, we have
 $$ \rho(\lambda)=\lim_{t\to \infty} \E{h_t(GSMPGW(\kappa,\beta))}=\E{h(GSMPGW(\kappa,\beta))} \, , $$
 where $h$ is now the probability that the contact process on the dynamical graph survives\footnote{Of course, the functional $h$ is not a continuous functional!} for all times $t\in \R_+$.
 \end{proof}

Let us stress that Theorem~\ref{thm:metastable_density} only holds in the slow extinction regime, which we defined as the regime in which $\theta(\lambda)>0$ for all $\lambda>0$. Although we do not prove this, it is natural to conjecture that the link between metastability and $\theta(\lambda)>0$ holds more generally. Stated differently, even in the fast extinction regimes (in the parameters $\gamma$ and $\eta$), we expect to observe a phase transition in $\lambda$ between a regime exhibiting metastability with metastable density $\theta(\lambda)>0$, and a regime with $\theta(\lambda)=0$ and subexponential extinction time for the infection.

\begin{conj}[Local survival is equivalent to metastability]\label{conj:LSmetastability}
	For every parameter $\gamma$, $\eta$, and $\lambda$, the following are equivalent:
	\begin{itemize}
		\item $\theta(\lambda)=0.$
		\item The extinction time is subexponential in $n$.  
	\end{itemize}
	Moreover, if $\theta(\lambda)>0$, then the mean extinction time is exponential in $n$ and there is metastability of the infection, with metastable density $\theta(\lambda)$.
\end{conj}

Finally, we mention~\cite{Jacob-Linker-Morters19,Jacob-Linker-Morters24}, where the authors study the contact process on vertex-updating inhomogeneous random graphs. We can adapt Theorem~\ref{thm:metastable_density} and Conjecture~\ref{conj:LSmetastability} to these settings. In particular, in the slow extinction regime, we have again metastability for all $\lambda>0$, with the metastable density now given by the survival probability of the contact process on the $VGSMPGW(\kappa,\beta)$, starting from only the root infected.



\appendix


\section{Coupling of Markov Processes}\label{appendix:coupling-MP}

 \begin{lem}[Coupling of Markov processes]\label{lem:coupling-Markov-process}
 Let $X$ and $Y$ be two continuous time Markov processes on a countable state space $E$ with transition rates $( \tau^X_{e,f} )_{e,f\in E}$ and $( \tau^Y_{e,f} )_{e,f\in E}$ respectively, started from same state $e_0\in E$.
 For $e\in E$, we define
 \begin{equation}
 \Delta(e):= \sum_{f\in E, \: f\neq e} | \tau^X_{e,f} - \tau^Y_{e,f} | \,.
 \end{equation}
 Then there exists a coupling $(\tilde X,\tilde Y)$ of $(X,Y)$ such that for all time $t\ge 0$,
 \begin{equation}
 \Prob{ \exists s\le t, \: \tilde X_s \neq \tilde Y_s }  \le t\ \E{\sup_{s\le t} \Delta(X_s)} \, .
 \end{equation}
 \end{lem}
 
 \begin{proof}
 We let $Z$ be a Markov process started from $e_0$ with transition rates $\tau^Z_{e,f} := \min( \tau^X_{e,f},\tau^Y_{e,f} )$, and constructed from a graphical representation of independent PPP $(\Pi^Z_{e,f})_{e,f\in E}$ with intensities $\tau^Z_{e,f}$.
 We also let $(\Pi^{Z,X}_{e,f})_{e,f\in E}$ and $(\Pi^{Z,Y}_{e,f})_{e,f\in E}$ be independent families of PPP with respectively intensities $\tau^X_{e,f} - \tau^Z_{e,f}$ and $\tau^Y_{e,f} - \tau^Z_{e,f}$ respectively.
 We now define $\Pi^X_{e,f} = \Pi^Z_{e,f}\cup\Pi^{Z,X}_{e,f}$ for $e,f\in E$, and observe these are independent PPP with intensities $\tau^X_{e,f}$.
 These can thus be used to construct a Markov process $\tilde X$ with same law as $X$. We similarly construct $\tilde Y$ with same law as $Y$.
 
 We now work conditionally on the process $(Z_s)_{s\ge0}$, and write $0=s_0<s_1<\dots<s_k=t$ its jumping times up to time $t$ and $e_0,\dots,e_{k-1}$ its successive states, so that $Z_s = e_i$ for $s_i\le s < s_{i+1}$. We also let $T:=\inf\{s\ge 0,\, \tilde X_s\ne Z_s \text{ or }\tilde Y_s\ne Z_s\}$.
 Letting the conditioning on $Z$ be implicit to enlighten the notation, we have, for $0\le i\le k-1$:
 \begin{align*}
 \Prob{ T\le s_{i+1} \mid T> s_{i} }
 & = 1-\Prob{ \bigg( \bigcup_{f\ne e_i} (\Pi^{Z,X}_{e_i,f} \cup \Pi^{Z,Y}_{e_i,f}) \bigg) \cap [s_i,s_{i+1}[ =\emptyset } \\
 & = 1-\exp\left( - (s_{i+1}-s_i) \Delta(e_i) \right)= 1-\exp\left( - (s_{i+1}-s_i) \Delta(\tilde X_{s_i}) \right)  \\
 &\le (s_{i+1}-s_i)\sup_{s\le t}\Delta (\tilde X_s).
 \end{align*}
 Summing over $i$, we obtain
 $$ \Prob{T\le t \mid Z}\le t \sup_{s\le t}\Delta (\tilde X_s) \,, $$
 and we conclude by taking the expectation of this inequality.
\end{proof}

\subsection*{Acknowledgement}

The authors would like to thank the anonymous referees, an Associate Editor and the Editor for their constructive comments that improved the quality of this paper.

The first author LD was supported by a InfoMaths PhD Fellowship (ED-512) and was carried out while he was an intern at Unit\'e de Math\'ematiques Pures et Appliqu\'ees, Lyon.
The second author EJ was supported by CNRS, by the LABEX MILYON (ANR-10-LABX-0070) of Universit\'e de Lyon and by GrHyDyn (ANR-20-CE40-0002).


\addcontentsline{toc}{section}{References}
\bibliographystyle{abbrv}
\bibliography{LocDynDortJacob_v2}

\bigskip
\bigskip

\begin{flushleft}
\textsc{L\'eo Dort}
 
\textsc{Unit\'e de Math\'ematiques Pures et Appliqu\'ees, \'Ecole Normale Sup\'erieure de Lyon}

\textit{Email:} {\color{Myred}\texttt{leo.dort@ens-lyon.fr}}

\bigskip

\textsc{Emmanuel Jacob}
 
\textsc{Unit\'e de Math\'ematiques Pures et Appliqu\'ees, \'Ecole Normale Sup\'erieure de Lyon}

\textit{Email:} {\color{Myred}\texttt{emmanuel.jacob@ens-lyon.fr}}
\end{flushleft}

\end{document}